\documentclass{article}
\usepackage[utf8]{inputenc}

\usepackage[style=alphabetic]{biblatex}
\usepackage{amsmath}
\usepackage{amsfonts}
\usepackage{amsthm}
\usepackage{amssymb}
\usepackage{mathrsfs}

\usepackage{sidecap}
\usepackage{overpic}
\usepackage{caption}
\usepackage{lipsum}
\usepackage{floatrow}
\usepackage{hyperref}

\usepackage{mathdots}

\title{Proper affine deformations of positive representations}
\author{Jean-Philippe Burelle\thanks{This author acknowledges the support of the Natural Sciences and Engineering Research Council of Canada (NSERC), [funding reference number RGPIN-2020-05557]} \and Ne\v{z}a \v{Z}ager Korenjak}

\newtheorem{prop}{Proposition}[section]
\newtheorem{lem}[prop]{Lemma}

\theoremstyle{definition}
\newtheorem{defn}[prop]{Definition}
\newtheorem{rem}[prop]{Remark}
\newtheorem{thm}[prop]{Theorem}
\newtheorem{question}[prop]{Question}

\newcommand{\bR}{\mathbb{R}}
\newcommand{\bH}{\mathbb{H}}
\newcommand{\mA}{\mathcal{A}}
\newcommand{\mC}{\mathcal{C}}

\newcommand{\SL}{\mathrm{SL}}
\newcommand{\GL}{\mathrm{GL}}
\newcommand{\PSL}{\mathrm{PSL}}
\newcommand{\SO}{\mathrm{SO}}
\newcommand{\Flag}{\mathrm{Flag}}
\newcommand{\diag}{\mathrm{diag}}
\newcommand{\bilin}{\cdot}
\newcommand{\half}{\mathcal{H}}

\newcommand{\ival}[2]{(\!(#1,#2)\!)}
\newcommand{\opp}{\mathrm{opp}}

\renewcommand{\vec}{\mathbf}

\newcommand{\sgn}{\mathrm{sgn}}

\begin{document}

\maketitle
\begin{abstract}
    We define for every positive Anosov representation of a nonabelian free group into $\SO(2n,2n-1)$ a family of $\bR^{4n-1}$-valued cocycles which induce proper affine actions on $\bR^{4n-1}$. We construct fundamental domains in $\bR^{4n-1}$ bounded by generalized crooked planes for these affine actions, and deduce that the quotient manifolds are homeomorphic to handlebodies.
\end{abstract}

\section{Introduction}
An affine manifold is a smooth manifold equipped with a flat, torsion-free affine connection. If it can be obtained by taking the quotient of the affine space $\bR^n$ by a discrete subgroup of affine transformations acting properly and pushing forward the trivial connection, it is called complete. The problem of classifying complete affine manifolds is therefore the same as the classification of proper actions on $\bR^n$ by discrete groups of affine transformations.

An action of a discrete group $\Gamma$ by affine transformations is determined by two pieces of data: a \emph{linear part} $\rho: \Gamma \to \GL(n,\bR)$ which is a representation of the group, and a \emph{translational part} $u: \Gamma \to \bR^n$ which is a $\rho(\Gamma)$-cocycle. If the action of $(\rho, u)(\Gamma)$ on $\bR^n$ is properly discontinuous, we will call $u$ a proper (affine) deformation of $\rho.$

Unlike complete Euclidean manifolds whose classification is well understood, the classification of both closed and non-compact affine manifolds is still mysterious in general. Even though affine manifolds are flat, their fundamental groups can be hyperbolic. Examples of properly discontinuous affine actions by a nonabelian free group on $\bR^3$ were first discovered by Margulis in \cite{Margulis1, Margulis2}, in response to Milnor \cite{Milnor} asking whether complete affine manifolds with nonabelian free fundamental group exist.

In fact, properly discontinuous affine actions by free groups on $\bR^3$ are now fully understood. Fried--Goldman \cite{friedgoldman} proved that the linear part of each such action must preserve a symmetric bilinear form of Lorentzian signature, and therefore defines a hyperbolic surface  $\bH^2 / \rho(\Gamma).$ Given such a linear part, Danciger--Guéritaud--Kassel in \cite{DGK} describe all cocycles $u: \Gamma \to \bR^3$ such that $(\rho(\Gamma), u(\Gamma))$ acts properly on $\bR^3$. They do this in terms of \emph{strip deformations} of the hyperbolic surface $\bH^2 / \rho(\Gamma).$ Their approach also allows them to construct fundamental domains for these actions which are bounded by \emph{crooked planes}, first introduced by Drumm in \cite{Drumm}. Crooked planes have been explored further for instance in \cite{DrummGoldman2, CDG,BCDG}. Though initially an object in three dimensions,  Burelle and Treib in \cite{bt2022} give a notion of crooked halfspaces in higher-dimensional spheres and projective spaces.

In higher dimensions, the story of complete affine manifolds is less complete. Abels--Margulis--Soifer and Smilga give some criteria to determine which Lie groups admit discrete free subgroups such that there exist translational parts giving rise to a proper affine action; see \cite{AMS, smi14, Smilga2016,Smilga2018,Smilga2022}. In \cite{Zager}, some higher-dimensional examples for Fuchsian groups are constructed, adapting ideas in \cite{DGK}. In \cite{Smilga2014}, Smilga also constructs fundamental domains called tennis ball domains for affine actions of free groups with strongly contracting linear part in $\SO(2n,2n-1).$

The goal of this paper is to sharpen Smilga's result for a specific class of linear parts: \emph{positive} representations of free groups into $\SO(2n,2n-1)$, in the sense of Fock and Goncharov, see \cite{FG, GW}. These representations are defined in terms of a specific choice of punctured surface, and have geometric properties which relate closely with the geometry at infinity of the surface and its fundamental group (see Section \ref{sec:positive_reps} for the definition).

\begin{thm}\label{thm:proper_defs_exist}
    Let $S$ be a closed surface with $k\ge 1$ punctures. Let $\rho: \pi_1(S) \to \SO(2n,2n-1)$ be a positive Anosov representation. Then, there exists a proper affine deformation of $\rho$.
\end{thm}

In contrast to Smilga's result, no quantitative hypothesis is made on the contraction strength of the representation $\rho$. To obtain this theorem, we define a generalized notion of the infinitesimal strip deformations used by Danciger--Guéritaud--Kassel. In this generalization, the \emph{flag intervals} in $\Flag^+(\bR^{4n-1})$ defined in \cite{bt2022} play the role of hyperbolic geodesics, and the infinitesimal strip deformations can be interpreted as an infinitesimal shrinking of flag intervals. Using results on the Margulis invariant by Goldman--Labourie--Margulis and Ghosh--Treib \cite{GLM, GhoshTreib}, we then prove that all cocycles obtained by this method yield proper affine actions.

Danciger and Zhang proved in 2019 that the analog of Theorem \ref{thm:proper_defs_exist} for \emph{closed} surfaces is false~\cite{dancigerzhang}: they showed that positive (or, equivalently, \emph{Hitchin}) representations of closed surface groups can never admit proper affine deformations.

Under a slightly stronger assumption on the arc system used to define the strip deformation (a condition we call ``overfilling"), we reprove properness via a different method not relying on the machinery of diffused Margulis invariants from \cite{GhoshTreib}. We construct fundamental domains and show that they tile all of affine space:

\begin{thm}
    All proper affine actions $(\rho,u)$ obtained from generalized infinitesimal strip deformations along an overfilling arc system admit fundamental domains bounded by crooked planes. The quotient manifolds $\bR^{4n-1}/(\rho,u)(\Gamma)$ are homeomorphic to handlebodies.
\end{thm}

Our construction depends on a choice of arcs and further on a choice of extension of the positive boundary map associated to the linear part. For a fixed choice, we do \emph{not} obtain all the possible proper affine deformations of the given linear part in dimensions larger than $3$. However, it is natural to ask
\begin{question}
    How does the deformation depend on these choices? Varying over all the choices, do we get all proper affine deformations of a given positive linear part?
\end{question}
If the answer to this question is yes, then in particular we would be able to construct fundamental domains for all proper affine deformations with given (free positive Anosov) linear part, showing that they are all \emph{tame} (the quotients are homeomorphic to the interior of a compact manifold with boundary).

By a dimension count we know that there are redundancies and that some choices of the extension of the boundary map will give conjugate cocycles. However, for $\SO(2,1),$ one can make some canonical choices (see \cite{DGK}) and obtain a parametrization of the space of proper cocycles. It is not clear whether there is an analog of this for generalized infinitesimal strip deformations leading to a parametrization of the cone of proper cocycles.

In Section \ref{positive_pingpong}, we recall the notion of positivity for tuples of flags in $\bR^{2n,2n-1}.$ We gather basic facts about positive representations of (non-compact) surface groups into $\SL(4n-1)$ and $\SO(2n,2n-1).$ This section is mostly for background and fixing notation. For more detailed proofs, see \cite{bt2022}.

In Section \ref{construction}, we define a generalized infinitesimal strip deformation -- a cocycle associated to an extended boundary map of a free group with an arc system. In Section \ref{margulisinvt}, we compute the Margulis invariant of such a cocycle and show that the cocycle determines a proper action of a free group on $\bR^{2n,2n-1}.$

In Section \ref{domains}, we use \emph{crooked halfspaces} to construct fundamental domains for actions from Section \ref{construction}. This gives an alternate proof of Theorem \ref{thm:proper_defs_exist}.

\section{Positive Ping-pong}\label{positive_pingpong}

Fix  the standard orientation on $\bR^{4n-1}$ given by the canonical basis. An \emph{oriented flag} $F$ in $\bR^{4n-1}$ is a $4n$-tuple of properly nested vector subspaces
\[0=F^{(0)} \subsetneq F^{(1)} \subsetneq \dots \subsetneq F^{(4n-2)} \subsetneq F^{(4n-1)} = \bR^{4n-1},\]
together with a fixed choice of orientation on $F^{(i)}$, so that the orientation on $F^{(4n-1)}$ agrees with that of $\bR^{4n-1}$. Equivalently, one can specify a choice of $1$-dimensional orientation on each of the quotients $F^{(i)}/F^{(i-1)}$ for $i=1,\dots,4n-1$. The group $\SL(4n-1,\bR)$ acts transitively on the set of oriented flags, and so we can endow it with a topology by identifying it with a coset space: $\Flag^+(\bR^{4n-1}):= \SL(4n-1,\bR)/B^0$, where $B^0$ is the subgroup of upper triangular matrices with positive entries on the diagonal.

Let $\bilin$ denote the symmetric bilinear form of signature $(2n,2n-1)$ given by the following matrix in the standard basis
\[J = \begin{pmatrix}
 0 & 0 & 0 & \dots & 0 & 0 & -1\\
 0 & 0 & 0 & \dots & 0 & 1 & 0\\
 0 & 0 & 0 & \dots & -1 & 0 & 0\\
 \vdots & \vdots & \vdots & \iddots & \vdots & \vdots & \vdots\\
 0 & 0 & -1 & \dots & 0 & 0 & 0\\
 0 & 1 & 0 & \dots & 0 & 0 & 0\\
 -1 & 0 & 0 & \dots & 0 & 0 & 0\\
\end{pmatrix}.\]
We will call a basis of $\bR^{4n-1}$ a $J$-basis if, in that basis, the bilinear form $\cdot$ has matrix $J$.

\begin{defn}\label{def:flag}
    
An \emph{oriented isotropic flag} is an oriented flag such that
\begin{enumerate}
    \item $F^{(i)} = (F^{(4n-i)})^\perp$ for $i=1,\dots,4n-1$, where the orthogonal is taken with respect to $~\bilin~$ and
    \item For any choice of positive vectors $v \in F^{(i+1)}/F^{(i)}$ and $u\in (F^{(i)})^\perp/(F^{(i+1)})^\perp$ with $0\le i\le 4n-2$, $(-1)^{i+1} u \bilin v$ is positive (this sign is actually independent of the choice of $u,v$).
\end{enumerate}
\end{defn}
The subgroup of $\SL(4n-1,\bR)$ preserving $\bilin$ will be denoted by $\SO(2n,2n-1)$. Denote by $\Flag^+(\bR^{2n,2n-1})$ the space of oriented isotropic flags.

\begin{defn}
    A pair of oriented flags $F,G$ is \emph{oriented transverse} if for all $0\le i\le 4n-1$ we have a direct sum
    \[F^{(i)}\oplus G^{(4n-1-i)} = \bR^{4n-1},\]
    and the induced orientation matches with the fixed orientation on $\bR^{4n-1}$.
\end{defn}

Because the dimension of the ambient space is odd, oriented transversality is a symmetric relation. The group $\SL(4n-1,\bR)$ acts transitively on pairs of oriented transverse flags (\cite{bt2022}, Lemma 2.19). Similarly, it follows from the next proposition that the group $\SO(2n,2n-1)$ acts transitively on pairs of oriented transverse isotropic flags.

\begin{prop}\label{prop:JBasis}
    A pair of oriented isotropic flags $F,G$ is oriented-transverse if and only if there exists a positively oriented $J$-basis $E = (e_1,\dots,e_{4n-1})$ such that
    \[F = \{0\} \subset \langle e_1 \rangle \subset \langle e_1,e_2 \rangle \subset \dots \subset \langle e_1,e_2, \dots, e_{4n-1} \rangle\]
    and
    \[G = \{0\} \subset \langle e_{4n-1} \rangle \subset \langle e_{4n-2},e_{4n-1} \rangle \subset \dots \subset \langle e_1,e_{2}, \dots, e_{4n-1} \rangle.\]
    Given a basis $E$, we denote this pair by $(F_E, F_{\widehat{E}})$, where $\widehat{E}$ is the \emph{opposite basis} $(e_{4n-1}, -e_{4n-2}, e_{4n-3}, \dots, -e_2,e_1)$. Note that $\widehat{E}$ is a $J$-basis if and only if $E$ is a $J$-basis.
\end{prop}
\begin{proof}
    By \cite{bt2022} Lemma 2.19, there exists a basis satisfying all the required properties except that of being a $J$-basis. The fact that $F,G$ are isotropic guarantees that we can rescale each basis vector by a positive scalar to make it into a $J$-basis, and this does not change the flags $F,G$ or their orientation.
\end{proof}

\begin{rem}
    Note that the notion of an opposite basis $\widehat E$ to a basis $E$ is defined for any basis, not just an isotropic one.
\end{rem}

\begin{defn}
    A triple of oriented flags $(F,G,H)$ is \emph{positive} if for every choice of $0 \le i,j,k\le 4n-1$ such that  $i+j+k = 4n-1$, we have the direct sum
    \[F^{(i)} \oplus G^{(j)} \oplus H^{(k)} = \bR^{4n-1},\]
    and the induced orientation matches with the fixed orientation on $\bR^{4n-1}$. In particular, note that in a positive triple every pair is oriented transverse.
\end{defn}

More generally, we say a $k$-tuple of oriented flags is positive if every ordered sub-triple is positive.

Positivity of triples of oriented flags satisfies the following properties (\cite{bt2022}, Proposition 3.7):
\begin{enumerate}
    \item If $(F,G,H)$ is positive, then $(H,F,G)$ is positive.
    \item If $(F,G,H)$ is positive, then $(H,G,F)$ is not positive.
    \item If $(F,G,H)$ and $(F,H,K)$ are positive, then $(F,G,K)$ is positive.
\end{enumerate}

These are the axioms of a (strict) \emph{partial cyclic order}. We will use the terms \emph{cyclically ordered} and \emph{positive} for tuples of flags interchangeably.

\begin{defn}
    Given two oriented transverse flags $F, G \in \Flag^+(\bR^{4n-1}),$ the \emph{interval} between them is 
    \[\ival{F}{G} = \{ H \in \Flag^+(\bR^{4n-1}) ~|~ (F, H, G) \text{ is positive}\}.\]
    The \emph{opposite} of an interval is defined by $\ival{F}{G}^\opp = \ival{G}{F}$.
\end{defn}

Flag positivity is classically defined in terms of totally positive matrices. A matrix is \emph{totally positive} if all its minors are strictly positive. A matrix is \emph{upper triangular totally positive} (resp. \emph{lower triangular totally positive}) if it is upper triangular (resp. lower triangular) and all of its minors which are not zero by triangularity are strictly positive.
\begin{prop}\label{LowerTriangularSpan}
    A triple of oriented flags $F,G,H$ is positive if and only if there exists an oriented basis $E$ of $\bR^{4n-1}$ and a unipotent lower-triangular totally positive matrix $U$ such that $F = F_E$, $H = F_{\widehat{E}}$
    and $G^{(i)}$ is the span of the first $i$ columns of $U$.
\end{prop}
\begin{proof}
    This is proven in Lemma 3.2 of \cite{bt2022}.
\end{proof}

\begin{lem}\label{lem:cyclic_triple_intervals}
    Let $C$ be a positive $k$-tuple of oriented flags, with $F,G\in C$ an adjacent pair. Let $X\in \ival{F}{G}$. Then, the $(k+1)$-tuple $\widehat{C}$ consisting of $C$ with the flag $X$ inserted between $F$ and $G$ is positive.
\end{lem}
\begin{proof}
    By cyclic invariance, we may assume that $G$ is the first flag of the tuple and $F$ is the last, and $X$ is inserted at the end.

    Let $H,K \in C$ be distinct and assume without loss of generality that $(G,H,K)$ and $(H,K,F)$ are positive (by switching $H$ and $K$ if needed). We must show that $(H,K,X)$ is positive.
    
    Since $(F,X,G)$ and $(F,G,H)$ are positive, we get that $(F,X,H)$ is positive. Then, since $(H,K,F)$ is positive, we get $(H,K,X)$ is positive as desired. 
\end{proof}

The following property of the cyclic order on oriented flags, which is called \emph{properness} in \cite{bt2022}, will be useful.

\begin{lem}[\cite{bt2022} Proposition 3.13]\label{lem:intervals_proper}
    If $(F,G,H,K)$ is a positive quadruple, then $\overline{\ival{G}{H}} \subset \ival{F}{K}$.
\end{lem}

\begin{prop}\label{prop:quadruplepositive}
    A quadruple of oriented flags $(F,G,H,K)$ is positive if and only if there exists an oriented basis $e_1,\dots,e_{4n-1}$ of $\bR^{4n-1}$ such that $F,G,H$ are as Proposition \ref{LowerTriangularSpan}, and $K$ is represented by $SLS$ where $L$ is lower triangular totally positive and $S = \diag(1,-1,1,-1,\dots,-1,1)$.
\end{prop}
\begin{proof}
    This follows from Lemma 3.10 of \cite{bt2022} and Proposition \ref{LowerTriangularSpan}.
\end{proof}

\begin{prop}[\cite{bt2022} Lemma 3.17]\label{prop:quadruplepositivemap}
    A quadruple of oriented flags of the form $(F_E, G, H, F_{\widehat{E}})$ is positive if and only if there exists a totally positive matrix $M \in \SL(4n-1)$  such that $M F_E = G$, $M F_{\widehat{E}} =  H$. If $F_E, G, H, F_{\widehat E}$ are all isotropic, we can choose $M$ totally positive in $\SO(2n,2n-1).$
\end{prop}

\subsection{Positivity in $\SO(2n,2n-1)$}
The definition of positivity was extended by Lusztig \cite{Lusztig1994} to reductive Lie groups. Since we are interested in orthogonal groups $\SO(2n,2n-1)$, we will show that with the bilinear form $J$ and the right choice of Cartan decomposition, positive matrices in $\SO(2n,2n-1)$ are also positive in $\SL(4n-1,\bR)$ (that is, they are totally positive matrices).

We now recall Lusztig's definition of the positive subsemigroup of a real split Lie group $G$. We fix a Cartan decomposition of $\mathfrak{g}$, a set of simple roots $\Sigma$ for some choice of positive roots, and a choice of generators $X_\alpha$ for the simple root spaces. Let $B$ be the Borel subgroup corresponding to the subalgebra generated by the Cartan subalgebra and the positive root spaces. Denote by $s_\alpha \in W$ the simple root reflection in the root $\alpha$ and let $s_{\alpha_1} s_{\alpha_2} \dots s_{\alpha_N}$ be a reduced expression for the longest element in $W$.

\begin{defn}
    The \emph{positive subsemigroup} of $B\subset G$ is the image of the map
    \begin{align*}
        (\bR^{>0})^N &\longrightarrow B \\
        (t_1,\dots,t_N) &\longmapsto e^{t_1 X_{\alpha_1}}e^{t_2 X_{\alpha_2}} \dots e^{t_N X_{\alpha_N}}
    \end{align*}
\end{defn}

We are interested in the groups $G = \SL(4n-1,\bR)$ and $G' = \SO(2n,2n-1)$. Consider the Cartan subgalgebra $\mathfrak{a}\subset \mathfrak{sl}(4n-1,\bR)$ consisting of diagonal matrices. We choose the Cartan subalgebra $\mathfrak{a}' \subset \mathfrak{so}(2n,2n-1)$ which also consists of diagonal matrices, so that $\mathfrak{a}' = \mathfrak{a}\cap \mathfrak{so}(2n,2n-1)$.

The collection of roots $\alpha_i \in \mathfrak{a}^*$ given by $\alpha_i(\diag(a_1,\dots,a_{4n-1})) = a_i - a_{i+1}$ for $i=1,\dots,4n-2$ forms a set of simple roots for $\mathfrak{sl}(4n-1)$. Similarly, $\alpha_i' = a_i - a_{i+1}$ for $1\le i \le 2n-1$ is a set of simple roots for $\mathfrak{so}(2n,2n-1)$.

Denoting by $E_{i,j}$ a $(4n-1)\times(4n-1)$ matrix with a $1$ in position $(i,j)$ and all other coefficients $0$, the root space associated to the simple root $\alpha_i$ is spanned by $E_{i,i+1}$. The root space for $\alpha'_i$ is spanned by $E_{i,i+1} + E_{4n-1-i,4n-i}$.

Denote by $s_i$ the simple root reflection in the root $\alpha_i$ and by $s'_i$ the simple root reflection in $\alpha'_i$. The longest word in the Weyl group of $\mathfrak{so}(2n,2n-1)$ can be written
\[w'_0 = ((s'_1 s'_3 \dots s'_{2n-1})(s'_2 s'_4 \dots s'_{2n-2}))^{2n-1}\]
and for $\mathfrak{sl}(4n-1,\bR)$
\[w_0 = (s_1 s_{4n-2} s_3 s_{4n-4} \dots s_{2n-1} s_{2n} s_{2n-1}s_2 s_{4n-3} s_4 s_{4n-5} \dots s_{2n-2} s_{2n+1})^{2n-1}.\]
Note that $w_0$ can be obtained from $w_0'$ by the replacing each instance of $s_i'$ by $s_i s_{4n-1-i}$ if $i<2n-1$ and $s_{2n-1}'$ by $s_{2n-1}s_{2n}s_{2n-1}$.

Now, since
\[e^{t(E_{i,i+1} + E_{4n-1-i,4n-i})} = e^{tE_{i,i+1}}e^{tE_{4n-1-i,4n-i}}\] for $i<2n-1$ and
\[e^{t(E_{2n-1,2n} + E_{2n,2n+1})} = e^{\frac t2 E_{2n-1,2n}}e^{t E_{2n,2n+1}}e^{\frac{t}{2}E_{2n-1,2n}},\]
the positive subsemigroup for $\SO(2n,2n-1)$ is contained in the positive subsemigroup for $\SL(4n-1,\bR)$. In particular, it consists of upper triangular totally positive matrices.

\subsection{Positive representations of free groups}\label{sec:positive_reps}
Let $\Gamma \subset \PSL(2,\bR)$ be a free, convex cocompact subgroup. Denote by $\Lambda_\Gamma \subset \partial_\infty \bH^2$ the limit set of $\Gamma$. It is an embedding of the boundary $\partial_\infty \Gamma$ of the group. The boundary $\partial_\infty \bH^2$ is a circle; orient it counter-clockwise, inducing an orientation on $\Lambda_\Gamma$ (and thus on $\partial_\infty\Gamma$). Topologically, the limit set is the Cantor set.

\begin{defn}
    A representation $\rho: \Gamma \rightarrow \SO(2n,2n-1)$ is \emph{positive Anosov} if there exists a $\rho$-equivariant map $\xi: \Lambda_\Gamma \to \Flag^+(\bR^{2n,2n-1})$ which maps positively oriented triples to positive triples. 
\end{defn}

\begin{rem}
    Note that the transversality condition of positivity forces $\xi$ to be injective, and that a positive Anosov representation is Borel Anosov. Our definition of positivity therefore differs from some authors who allow the image of the positive map $\xi$ to be a topological circle (which allows generalized cusps). From here on out, we will use the term \emph{positive} to refer to positive Anosov representations.
\end{rem}

On the surface $S =\bH^2 / \Gamma,$ there exist pairwise disjoint properly embedded geodesic arcs $a_1, \ldots, a_N$ that cut $S$ into a topological disk. Moreover, we can choose a pair of lifts $a_i{}_+, a_i{}_-$ to $\bH^2$ of each arc $a_i$ so that they bound a fundamental domain for $\Gamma$ in $\bH^2.$ We can choose elements $\gamma_i \in \Gamma$ such that $\gamma_i(a_i{}_-) = a_i{}_+,$ and $\gamma_1, \ldots, \gamma_N$ generate $\Gamma.$ We can mimic this construction in the space of oriented isotropic flags:

\begin{prop}\label{prop:intervals}
    If $\rho: \Gamma \to \SO(2n,2n-1)$ is a positive representation, then there exist intervals $I_1^\pm, \dots, I_N^\pm \subset \Flag^+(\bR^{4n-1})$ with all endpoints isotropic flags and satisfying:
    
    \begin{enumerate}
        \item If $J_1,J_2 \in \{I_1^-,\dots,I_N^-,I_1^+,\dots,I_N^+\}$ are distinct, $\overline{J_1} \subset J_2^\opp$.
        \item $\rho(\gamma_j)(I_j^{-\opp}) = I_j^+$.
    \end{enumerate}
\end{prop}

\begin{proof}
    Let $\xi$ be a positive $\rho$-equivariant map.

    Choose properly embedded arcs $a_i{}_\pm \subset \bH^2$ cutting the surface into a disk as above. Denote the endpoints of each geodesic $a_i{}_-$ by $x_i,y_i$ and order them so that the positive intervals $\ival{x_i}{y_i}$ are pairwise disjoint. For each $i$, the point $x_i \in \partial_\infty \bH^2$ lies in a unique complementary interval $(c_i,d_i)$ to the Cantor set $\partial_\infty \Gamma \subset \partial_\infty \bH^2$. Choose oriented isotropic flags $X_i \in \ival{\xi(c_i)}{\xi(d_i)}$ one by one, shrinking the intervals in which $X_i$ is picked as needed to preserve the positive ordering of the Cantor set, the points $X_i$, and their images $\rho(\gamma_i)X_i$. Choose oriented isotropic flags $Y_i$ using the same process on the points $y_i$.

    Define the intervals $I_i^-:= \ival{X_i}{Y_i}$ and $I_i^+:= \ival{\rho(\gamma_i)X_i}{\rho(\gamma_i)Y_i}$, so that point 2 is satisfied.

    By positivity of the equivariant map $\xi$ and Lemma \ref{lem:cyclic_triple_intervals}, the endpoints of the intervals defined this way are cyclically ordered. This fact, together with Lemma \ref{lem:intervals_proper}, implies point 1.
\end{proof}

This Ping-pong description of the dynamics of a positive representation was used in \cite{bt2022} to define fundamental domains for the action of these representations on projective spaces and spheres. We will adapt these ideas to the affine setting in Section \ref{domains}.

\section{Constructing proper cocycles}\label{construction}
In this section, we construct a $\rho(\Gamma)-$cocycle $u: \Gamma \to \bR^{2n,2n-1}$ by explicitly describing $u(\gamma)$ for each $\gamma.$ The construction is inspired by Danciger--Guéritaud--Kassel strip deformations \cite{DGK} (see also \cite{Zager}).

Let $\Gamma\subset \PSL(2,\bR)$ be a free convex cocompact subgroup, and denote respectively the associated surface by $S=\bH^2/\Gamma$ and its compactification by $\overline{S} = (\bH^2 \cup \partial_\infty \bH^2)/\Gamma$. Recall that a properly embedded transversely oriented arc $a$ on $S$ is an arc with no self-intersections exiting the convex core of $S$ on both sides, and endowed with a notion of a positive and negative direction of crossing it, or equivalently with the notion of ``left" and ``right" along $a$. Moreover, we will require that all our arcs extend continuously to $\overline{S}$ and do not self-intersect at the boundary.

Let $\rho: \Gamma\to \SO(2n,2n-1)$ be a positive representation and $\xi$ its associated boundary map. Denote by $\mathcal{A}$ a finite system of properly embedded pairwise disjoint transversely oriented arcs  on $S$, with pairwise disjoint extensions to $\overline{S}$. Let $\Lambda_{\mA} \subset \partial_\infty \bH^2$ be the set of endpoints of the lifts of all the arcs in $\tilde\mA$ to the universal cover $\bH^2.$ The left endpoint of $a \in \tilde\mA$ will be called $a^+$ and the right endpoint $a^-.$ Extend $\xi$ to a positive equivariant map from $\Lambda_\Gamma \cup \Lambda_\mA$ to $\Flag^+(\bR^{2n,2n-1}),$ and call it $\xi$ as well. This is possible by imitating the proof of Proposition \ref{prop:intervals}, adding a representative of each endpoint in the correct interval using Lemma \ref{lem:cyclic_triple_intervals} (one by one), and extending by equivariance. The action on intervals bounded by $\xi(a^\pm)$ then mimics the action of $\Gamma$ on the intervals determined by transversely oriented arcs in $\partial_\infty\bH^2.$

Choose  positively oriented vectors
 $$\vec v_a^+ \in \xi(a^+)^{(1)}$$ and $$\vec v_a^- \in \xi(a^-)^{(1)}$$
for each $a \in \tilde \mA$, satisfying the equivariance condition $\vec{v}_{\gamma \cdot a}^\pm = \rho(\gamma)\vec{v}^\pm_a$.

Choose a lift $p_0 \in \bH^2 \smallsetminus \tilde \mA$ of the basepoint of $\pi_1(S)\cong\Gamma$, and let $\gamma \in \Gamma.$ Let  $c_\gamma: [0,1] \to \bH^2$ be a path with $c_\gamma(0) = p_0$ and $c_\gamma(1) = \gamma \cdot p_0$ such that $c_\gamma$ intersects each arc in $\tilde\mA$ transversely and does not self-intersect within any arc in $\tilde\mA$. Denote by $X_{c_\gamma} $ the set of all intersection points between the path $c_\gamma([0,1])$ and the arcs in $\tilde\mA.$ For $x \in X_{c_\gamma},$ denote by $a_x$ the arc in $\tilde\mA$ intersected by $c_\gamma([0,1])$ at $x.$
Moreover, let $\sigma(x) = 1$ if the path $c_\gamma$ intersects the arc $a_x \in \tilde \mA$ positively at $x$, and $\sigma(x) = -1$ if the intersection is negative. 

\begin{defn}\label{def:cocycle}
    A \emph{generalized infinitesimal strip deformation} of a positive representation $\rho: \Gamma \to \SO(2n,2n-1)$ associated to an arc system $\mA$ 
    is a map $u:\Gamma \to \bR^{2n,2n-1}$ defined by

\begin{equation} \label{CocycleDef}
    u(\gamma) = \sum_{x \in X_{c_\gamma}}\sigma(x)(
    \vec{v}_{a_x}^+ - 
    \vec{v}_{a_x}^-).
\end{equation}
\end{defn}
An illustration of this construction can be seen in Figure \ref{path}. Intuitively, we are adding a vector to $u(\gamma)$ each time $\gamma$ crosses an arc in $\mA.$

\begin{figure}[h!]
 \floatbox[{\capbeside\thisfloatsetup{capbesideposition={right,top},capbesidewidth=0.62\textwidth}}]{figure}[\FBwidth]
  {\caption{Choose $\rho(\Gamma)$-equivariant positively oriented vectors $\vec v_a^\pm \in \xi(a^\pm)^{1},\vec v_b^\pm \in \xi(b^\pm)^{1},\vec v_c^\pm \in \xi(c^\pm)^{1}.$ In the example drawn, the curve $c_\gamma$ crosses the arcs $a$ and $c$ positively and the arc $b$ negatively, so $\sigma(a\cap c_\gamma([0,1])) = 1$ and  $\sigma(c\cap c_\gamma([0,1])) = 1$, while   $\sigma(b\cap c_\gamma([0,1])) = -1.$ We would associate to $\gamma$ the translation vector $u(\gamma) =  \vec v_a^+ -  \vec v_a^- -(\vec v_b^+ - \vec v_b^-) + \vec v_c^+ - \vec v_c^-.$}\label{path}}
{\begin{overpic}[width=0.34\textwidth]{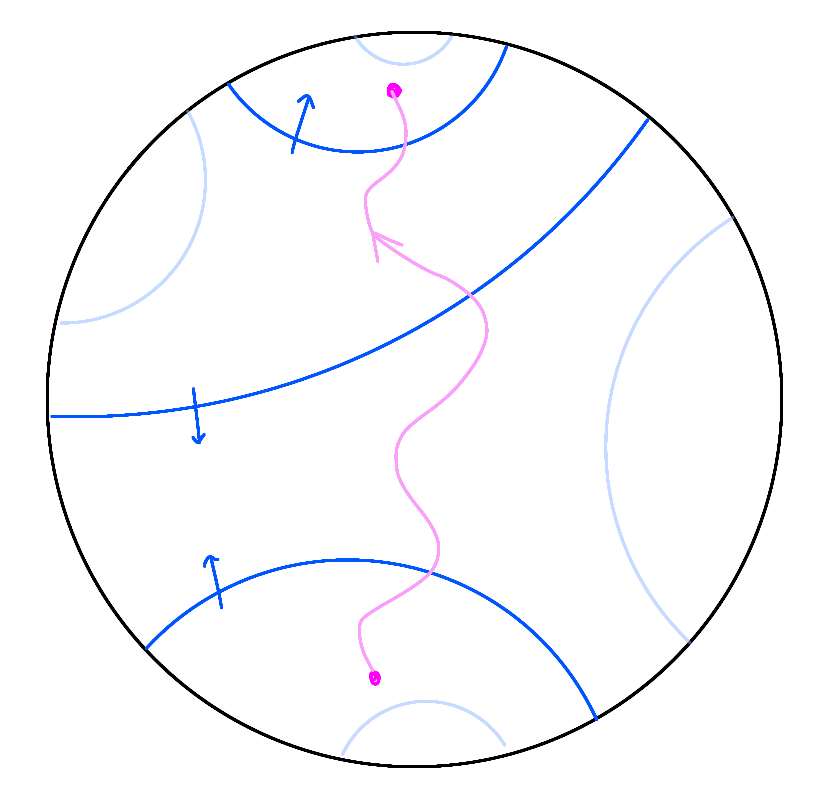}
\put(48,11){\tiny $p_0$}
\put(50, 82){\tiny $\gamma \cdot p_0$}
\put(60,50){\tiny $c_\gamma$}
\put(30,30){\tiny $a$}
\put(30,43){\tiny $b$}
\put(32,76){\tiny $c$}
\put(11,13){\tiny $a^+$}
\put(71,5){\tiny $a^-$}
\put(22,88){\tiny $c^+$}
\put(61,92){\tiny $c^-$}
\put(-2,44){\tiny $b^-$}
\put(82,82){\tiny $b^+$}
\end{overpic}}
\end{figure}

\begin{rem}
    While the set $X_{c_\gamma}$ can depend on the choice of $c_\gamma$, the sum defining $u(\gamma)$ does not.
\end{rem}

\begin{lem}\label{well-defined}A generalized infinitesimal strip deformation is a cocycle.
\end{lem}
\proof We need to show that $u(\gamma_1 \gamma_2) = \rho(\gamma_1)u(\gamma_2) + u(\gamma_1).$ 
Let $c_{\gamma_1}$ be a path defining $u(\gamma_1)$ and let $c_{\gamma_2}$ be a path defining $u(\gamma_2).$ Define $$c_{\gamma_1\gamma_2}(t) =    \left\{
\begin{array}{ll}
      c_{\gamma_1}(2t) & t \leq \frac12 \\
      
      \gamma_1 \cdot c_{\gamma_2}(2t-1) & t \ge \frac12 \\
\end{array} 
\right. $$
As $c_{\gamma_1}(1) = \gamma_1 \cdot p_0 = \gamma_1 \cdot c_{\gamma_2}(0),$ the path $c_{\gamma_1 \gamma_2}$ is continuous.  We have $c_{\gamma_1\gamma_2}(0)=p_0$ and $c_{\gamma_1\gamma_2}(1) = \gamma_1 \cdot (c_{\gamma_2}(1)) = \gamma_1\gamma_2 \cdot p_0,$ so $c_{\gamma_1\gamma_2}$ really does define $u(\gamma_1\gamma_2).$ It follows from equivariance of the choices of  $\vec{v}_a^\pm$ and  the definition of $u(\gamma_1\gamma_2)$ when using the path $c_{\gamma_1\gamma_2}$ that the cocycle identity holds.

\endproof

\begin{rem}This construction depends on the choice of base ``tile" (connected component of $\bH^2$ with all arcs of $\tilde \mA$ removed) in which $p_0$ lies. However, different choices of the base tile correspond to cohomologous cocycles. The cocycle $u$ only depends on the arcs in $\mA$ up to homotopy relative to the boundary. Alternatively, we can always work with the unique geodesic representatives of the arcs.
\end{rem}

\begin{rem}
    In the case when $\rho$ is a Fuchsian representation, that is, when $\rho$ is the restriction to $\Gamma$ of the irreducible representation $\sigma_{4n-1}: \SL_2(\bR) \to \SL_{4n-1}(\bR)$ to $\Gamma$, we can obtain these cocycles as linear combinations of limits of \emph{higher strip deformations} from \cite{Zager} when we send the ``waists" to $\infty.$ For $n=1,$ our construction recovers \emph{infinitesimal strip deformations} from \cite{DGK}.
\end{rem}

\section{The Margulis invariant}\label{margulisinvt}

In this section, we recall the Margulis invariant and compute it for generalized infinitesimal strip deformations as defined in the previous section. In Theorem \ref{thm:properness}, we show that the Margulis invariant of these cocycles is uniformly positive in order to conclude by a theorem of Ghosh--Treib (\cite{GhoshTreib}) that the actions they define are proper.

\begin{defn} We say that an element $g\in \SO(2n,2n-1)$ is \emph{regular} if it is diagonalizable with distinct real eigenvalues.
\end{defn}
A regular element has an attracting isotropic flag $g^+$ and a repelling isotropic flag $g^-$, obtained respectively by taking the span of the first $k$ eigenspaces in decreasing (resp. increasing) order of eigenvalue modulus.

Since real eigenvalues in $\SO(2n,2n-1)$ come in inverse pairs and there is an odd number of them,
regular elements always admit a \emph{neutral eigenvector} of eigenvalue $1$.

Moreover, we can choose a canonical representative on this neutral eigenline. We can orient the attracting flag $g^+$ of $g$, thus also orienting the neutral line. The following lemma proves that the orientation of the neutral line does not depend on the choice of oriented flag $g^+$, and we will call this the positive orientation of the neutral line.

\begin{lem}\label{lem:orientflags}
    Let $G, \, H \in \Flag^+(\bR^{2n,2n-1})$ be two  oriented isotropic flags with the same underlying unoriented flag. Then the orientations on the quotients $H^{(2n)}/H^{(2n-1)}$ and $G^{(2n)}/G^{(2n-1)}$ agree.
\end{lem}

\proof Let $E = (e_1, e_2, \ldots, e_{4n-1})$ be a positive $J$-basis defining $G,$ that is, $e_i$ is positive in $G^{(i)}/G^{(i-1)}.$ For $i = 1, \ldots, 4n-1,$ let $\sigma_i \in \{+1, -1\}$ and let $(\sigma_1e_1, \sigma_2 e_2, \ldots, \sigma_{4n-1}e_{4n-1})$ be a positive $J$-basis for $H.$ We want to show that $\sigma_{2n} = 1.$


Suppose $\sigma_i = -1$ for some $i \neq 2n$. Then by Condition 2 of Definition \ref{def:flag}, $\sigma_{4n-i}$ also has to be $-1$. Therefore, an even number of $\sigma_i$ for $i \neq 2n$ are negative. This forces $\sigma_{2n}$ to be $1$ in order for the basis
 $(\sigma_1e_1, \sigma_2 e_2, \ldots, \sigma_{4n-1}e_{4n-1})$ to be a positive basis.
\endproof

\begin{defn}
    The \emph{neutral vector} of a positive regular element $g \in \SO(2n,2n-1)$, denoted by $\vec x^0(g),$ is the positively oriented eigenvector of $g$ with eigenvalue $1$, normalized so that $\vec x^0(g) \cdot \vec x^0(g) = 1.$ 
\end{defn}

We can also define a neutral vector from just a pair of isotropic flags:

\begin{defn}\label{def:neutral_vec}
    The \emph{neutral vector} of a pair of transverse isotropic flags $(F,G)$, denoted $\vec{x}^0(F,G)$, is $e_{2n}$ where $(e_1,\dots,e_{4n-1})$ is any positive $J$-basis such that $F=F_E$ and $G=F_{\widehat{E}}$.
\end{defn}

This neutral vector exists by Proposition \ref{prop:JBasis} and is unique by Lemma \ref{lem:orientflags}. The two definitions are related by $\vec{x}^0(g) = \vec{x}^0(g^+,g^-)$.

\begin{defn}
Let $\rho: \Gamma \to \SO(2n,2n-1)$ be a representation with $\rho(\gamma)$ regular for all $\gamma \in \Gamma \smallsetminus \{1\}$ and let $ u: \Gamma \to \bR^{2n,2n-1}$ be a $\rho(\Gamma)$-cocycle. The \emph{Margulis invariant} of $\gamma \in \Gamma$ is 
    \[\alpha_{u}(\gamma) = u(\gamma) \bilin \vec{x}^0(\rho(\gamma)).\]
\end{defn}

The invariant was first introduced by Margulis in \cite{Margulis1, Margulis2}, where he showed the \emph{opposite sign lemma}:
\begin{lem}[Margulis, \cite{Margulis1, Margulis2}] Let $\Gamma < \SO(2,1)$ be a convex cocompact free group and $u:\Gamma \to \bR^{2,1}$ a $\Gamma$-cocycle. If there exist $\gamma, \, \eta \in \Gamma \smallsetminus \{1\}$ such that $\alpha_u(\gamma)\alpha_u(\eta) \leq 0,$ then $(\Gamma, u)$ does not act properly on $\bR^{2,1}$.
\end{lem}

In \cite{Lab01} Labourie extended the Margulis invariant to the diffused Margulis invariant, taking as an argument a geodesic current instead of a closed curve. Later, Goldman--Labourie--Margulis \cite{GLM} used this diffused version to show a converse of Margulis' opposite-sign lemma; an affine action with linear part in the image of the irreducible representation $\SO(2,1) \to \SO(2n,2n-1)$ is proper exactly when its diffused Margulis invariant is bounded away from $0$.

This result was extended by Ghosh--Treib in \cite{GhoshTreib} for representations with Anosov linear part, and it is their result that we will use here to show properness of our actions. Though not the most general restatement of the result of that paper, the formulation most convenient for us is:

\begin{prop}[\cite{GhoshTreib}]\label{GT} Let $\rho: \Gamma \to \SO(2n,2n-1)$ be Borel Anosov and denote by $t(\gamma)$ the translation length of $\gamma \in \Gamma < \PSL(2, \bR).$  Let $u$ be a $\rho(\Gamma)$-cocycle such that $\frac{\alpha_u(\gamma)}{t(\gamma)} \ge c > 0 $ for all $\gamma \in \Gamma \smallsetminus \{1\}.$ Then the affine action on $\bR^{2n,2n-1}$ defined by $u$ is proper.
\end{prop}

In the following three technical lemmas, we prove a higher dimensional analog of the fact that whenever a loop $\gamma \in \Gamma$ crosses an arc positively, there is a positive contribution to its Margulis invariant (Lemma \ref{Positivity}). Then, we find some lower bounds on the size of this contribution  (Lemmas \ref{lem:middle_entry_big}, \ref{bounded}). This allows us to conclude the main result of this section, that generalized infinitesimal strip deformations are proper (Theorem \ref{thm:properness}), by using Proposition \ref{GT}.

\begin{lem}\label{Positivity}
    Let $g \in \SO(2n,2n-1)$ be a positive regular element and fix an orientation on $g^+$ and $g^-$ such that the pair is oriented transverse. Let $F,G$ be oriented flags such that the quadruple $g^+,F,g^-,G$ is positive. Let $\vec{v}$ be any vector in the cone spanned by the positive half-lines in $F^{(1)}$ and $-G^{(1)}$. Then, $\vec v \cdot \vec x^0(g) > 0$.
\end{lem}
\begin{proof}
    Choose a $J$-basis $\vec{e}_1,\dots,\vec{e}_{4n-1}$ adapted to the quadruple as in Proposition \ref{prop:quadruplepositive}. Then, a positively oriented representative vector $\vec{u}$ for $F^{(1)}$ has strictly positive coordinates and any positively oriented representative vector $\vec{w}$ for $G^{(1)}$ has signs which alternate, starting with positive. In particular, the $2n$th coordinate of $G^{(1)}$ is negative, so any linear combination $\vec{v} = a\,\vec{u} - b\,\vec{w}$ with $a,b>0$ has a positive $2n$th coordinate.

    Since the basis $e_1, \ldots, e_{4n-1}$ is adapted to $g^+,F,g^-,G,$ the neutral eigenvector for $g$ is $\vec{e}_{2n}$. Then we have
    \[\vec{v}\cdot \vec x^0(g) = a\,\vec{e}_{2n}\bilin \vec{u} - b\,\vec{e}_{2n}\bilin\vec{w} > 0.\]
\end{proof}

\begin{lem}\label{lem:middle_entry_big}
    Let $M\in \SO(2n,2n-1)$ be a totally positive matrix. Then, $M_{2n,2n} > 1$.
\end{lem}
\begin{proof}
    Denote the canonical basis by $E$.
    By total positivity of $M$, the oriented isotropic flags $F_E$, $MF_E$, $MF_{\widehat{E}}$, $F_{\widehat{E}}$ are cyclically ordered.

    This implies that there exists a unipotent lower triangular totally positive matrix $L \in \SL(4n-1)$ such that $LF_E = MF_E$ (\cite{bt2022}, Lemma 3.2). In fact, $L \in \SO(2n,2n-1)$ since it maps a pair of isotropic flags to a pair of isotropic flags and is unipotent. By invariance of the cyclic order, the triple $F_E,L^{-1} M F_{\widehat{E}}, F_{\widehat{E}}$ is cyclically ordered and so there exists a unipotent upper triangular totally positive matrix $U$ such that $U F_{\widehat{E}} = L^{-1} M F_{\widehat{E}}$.

    We therefore have $LU F_E = MF_E$ and $LU F_{\widehat{E}} = MF_{\widehat{E}}$, and so $M^{-1}LU = D$ stabilizes $F_E$ and $F_{\widehat{E}}$. This implies that $D$ is diagonal with positive entries, and its middle entry is $1$. Since $LU=MD$, the middle entry of $M$ is equal to $(LU)_{2n,2n}$. Finally,
    \[M_{2n,2n} = (LU)_{2n,2n} = 1 + \sum_{k=2n+1}^{4n-1} L_{2n,k}U_{k,2n} > 1.\]
\end{proof}

\begin{lem}\label{bounded}
       Let $g$ and $h$ be  regular elements in $\SO(2n,2n-1),$
       and let $F$ be a flag, with $f \in F^{(1)}$ positively 
       oriented. Suppose $h^+, g^+, F, g^-, h^-$ are cyclically 
       ordered for some choice of orientations on $g^\pm$ and $h^\pm$. Then $\vec{x}^0(g) \cdot f < \vec{x}^0(h) \cdot f.$ 

       Similarly, if $g^-, h^-, F, h^+, g^+$ is cyclically ordered, $\vec{x}^0(h)\cdot f > \vec{x}^0(g)\cdot f.$
\end{lem}
\proof 
Let $E$ be a $J$-basis such that $h^+ = F_E$ and $h^- = F_{\widehat{E}}$.
By Proposition \ref{prop:quadruplepositivemap} there exists $M\in \SO(2n,2n-1)$, totally positive, such that $Mh^+ = g^+$ and $Mh^- = g^-$.

With this choice of basis, $\vec x^0(h) = \vec{e}_{2n}$ is the neutral eigenvector for $h$ and so $M\vec{e}_{2n}$ is the neutral eigenvector for $g$. Let $F' = M^{-1}F$. Since $M\ival{h^+}{h^-} = \ival{g^+}{g^-}$ and $F\in \ival{g^+}{g^-}$, preservation of the partial cyclic order by $\SO(2n,2n-1)$ implies $F'\in \ival{h^+}{h^-}$. This in turn implies that $f' = F'_1$ has only positive entries in the basis $E$ by Proposition \ref{LowerTriangularSpan}. We then have

\[\vec x^0(h) \bilin f = f_{2n} = (Mf')_{2n} = \sum_k M_{2n,k} f'_k > f'_{2n}\]

since $M_{2n,2n} > 1$ by Lemma \ref{lem:middle_entry_big}. Moreover, $f'_{2n} = e_{2n} \bilin f' = M e_{2n} \bilin M f' = \vec x^0(g) \bilin f$.
\endproof

Recall that we call a set of properly embedded disjoint arcs on a surface \emph{filling} if the arcs cut the surface into topological disks. In particular, every homotopically nontrivial curve on the surface intersects at least one of the arcs in a filling arc system. 

Using the previous three lemmas, we now show that the Margulis invariant is uniformly positive. We do this by observing that the number of intersections of a closed curve with a filling arc system grows linearly in the curve's length, and using Lemma \ref{bounded} to find the smallest possible contribution to the Margulis invariant at each intersection.

\begin{thm}\label{thm:properness}
    Let $\rho: \Gamma \to \SO(2n,2n-1)$ be a positive Anosov representation of a convex cocompact free group $\Gamma < \PSL_2(\bR)$ and let $\mA$ be a filling system of arcs on $\bH^2/\Gamma.$ Let $\xi: \Lambda_\Gamma \cup \Lambda_\mA \to \Flag^+(\bR^{2n,2n-1})$ be a positive equivariant map. Let $u : \Gamma \to \bR^{2n,2n-1}$ be a generalized infinitesimal strip deformation associated to $\mA.$
    Then $(\rho,u)(\Gamma)$ acts properly discontinuously on $\bR^{2n,2n-1}.$
\end{thm}
\proof By Proposition \ref{GT}, it is enough to show that the normalized Margulis invariant $\frac{\alpha_u(\gamma)}{t(\gamma)}$ is uniformly bounded below away from $0$.

Let $\gamma \in \Gamma \smallsetminus \{1\}.$ Recall the definition of $u(\gamma)$ from Equation \eqref{CocycleDef}:

\[u(\gamma) = \sum_{x \in X_{c_\gamma}}\sigma(x)(\vec{v}_{a_x}^+ - \vec{v}_{a_x}^-).\]

Let $a$ be an arc crossed by a path $c_\gamma$ in a positive direction. We need to bound $\vec x^0(\gamma)\cdot u(\gamma),$ from below, for which we start by determining the sign of inner products of the form $\vec x^0(\gamma)\cdot (
\vec v_a^+ - 
\vec v_a^-)$. The vector $\vec v =
\vec v_a^+ - 
\vec v_a^-$ is in the cone spanned by the positive vectors in $\xi(a^+)^{(1)}$ and $-\xi(a^-)^{(1)}$, and due to the positivity of $\xi$ and because $\gamma$ intersects $a$ in the positive direction, $\rho(\gamma)^+, \xi(a^+), \rho(\gamma)^-, \xi(a^-)$ is a positive quadruple of flags. Therefore by Lemma \ref{Positivity}, $\vec x^0(\gamma) \cdot \vec v > 0.$

Similarly, if $c_\gamma$ crosses $a$ in the negative direction, the product $\vec x^0(\rho(\gamma)) \cdot (
\vec v_a^+ - 
\vec v_a^-)$ is negative and this sign is cancelled by $\sigma(x)$.

As $u(\gamma)$ is a sum of vectors of the form $\vec v_a^+ - \vec v_a^-,$ we conclude that $\alpha_u(\gamma) > 0$ for all $\gamma \ne 1.$

We now show that for each $\gamma,$ the contribution to its Margulis invariant upon crossing each arc is uniformly bounded below.

Suppose $\gamma$ is not in the boundary of the convex core and that 
$c_\gamma$ crosses the arc $a$ positively. The arc $a$ crosses exactly two lifts of the boundary of the convex core of $S,$ call them $c^+_a$ and $c^-_a,$ with $c^+_a$ being closer to $a^+.$ Let $g_a \in \Gamma$ be the primitive element translating along $c_a^+$ in the positive direction with respect to the transverse orientation on $a$, and let $h_a$ be the primitive element translating positively along $c_a^-.$ An illustration of such a configuration is in Figure \ref{configuration}

\begin{figure}[h!]
 \floatbox[{\capbeside\thisfloatsetup{capbesideposition={right,top},capbesidewidth=0.62\textwidth}}]{figure}[\FBwidth]
  {\caption{We can see the configuration of endpoints of an arc $a$ (in blue) in $\tilde \mA$ and an element $\gamma$ that crosses it (in pink). The red curves are the lifts of two boundary curves, with elements $g_a$ and $h_a$ translating along them. The points $a^+,g_a^-,\gamma^-,h_a^-,a^-,h_a^+,\gamma^+,g_a^+$ are oriented positively, and therefore so are their images under the boundary map $\xi$, giving a positive configuration of eight flags in $\Flag^+(\bR^{4n-1}).$}\label{configuration}}
{\begin{overpic}[width=0.34\textwidth]{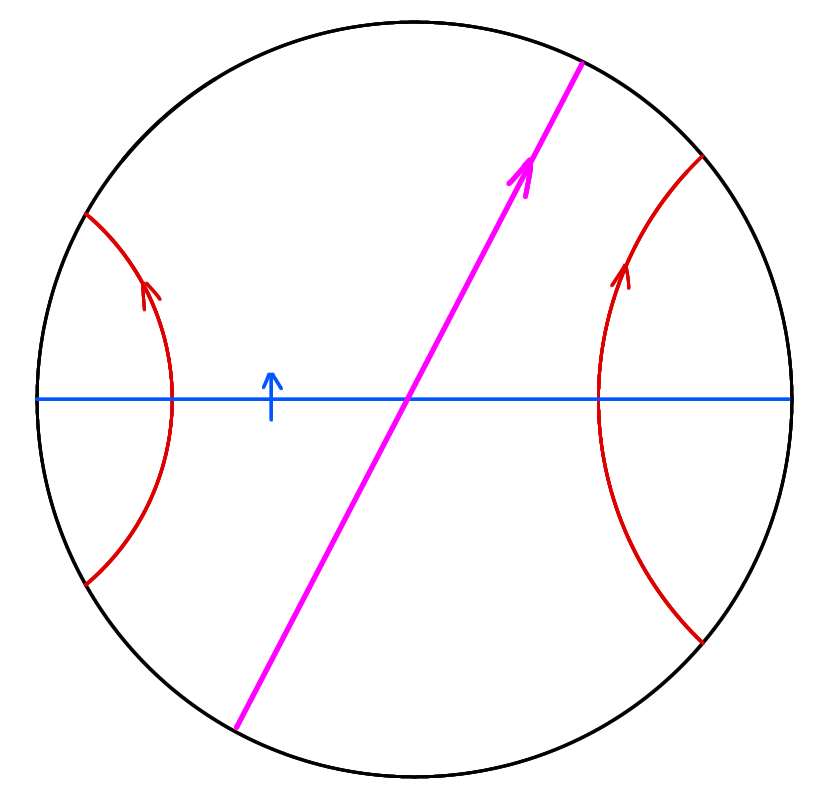}
\put(40,25){\tiny$\gamma$}
\put(71, 90){\small$\gamma^+$}
\put(23,0){\small$\gamma^-$}
\put(34,50){\tiny $a$}
\put(-4,47){\small$a^+$}
\put(97, 47){\small$a^-$}
\put(20,61){\tiny$g_a$}
\put(70,23){\tiny$h_a$}
\put(3,72.5){\small$g_a^+$}
\put(2,19){\small$g_a^-$}
\put(88,79){\small$h_a^+$}
\put(87,12){\small$h_a^-$}
\end{overpic}}
\end{figure}

Then $\xi(g_a^+), \xi(a^+), \xi(g_a^-), \xi(\gamma^-), \xi(\gamma^+)$ and $\xi(h_a^-), \xi(a^-), \xi(h_a^+), \xi(\gamma^+), \xi(\gamma^-)$ are both positive, as the axis of $\gamma$ has to lie inside the convex core of $S$ and thus separates $c_a^+$ and $c_a^-$. From Lemma \ref{bounded}, we get 
\begin{eqnarray*}
    \vec x^0(\gamma)\cdot \vec v &=& 
    \vec x^0(\gamma)\cdot \vec v_a^+ - 
    \vec x^0(\gamma) \cdot \vec v_a^- >  
    \\ 
    &>& 
    \vec x^0(g_a)\cdot \vec v_a^+ - 
    \vec x^0(h_a) \cdot \vec v_a^-.
\end{eqnarray*}

If $\gamma$ is one of the boundary curves of the convex core, then the above inequality still holds for the same reason, although one of the inequalities coming from Lemma \ref{bounded} is replaced by an equality.

Since the system of arcs $\mA$ is finite, the set $\{
\vec x^0(g_a)\cdot \vec v_a^+ - 
\vec x^0(h_a) \cdot \vec v_a^-\}_{a \in \mA}$ is a finite set of positive numbers, so it attains a minimum $c > 0$.

Because the holonomy of the surface $S$ is convex cocompact, its convex core has bounded diameter $D$. The translation length of an element $\gamma \in \Gamma < \PSL_2(\bR)$ is the same as the length of its geodesic representative, which stays within the convex core of $S.$ Therefore, if $\gamma$ crosses $n$ arcs in $\mA$, its translation length is at most $nD,$ and conversely, a closed geodesic of length $l$ must cross $\mA$ at least $\lfloor \frac lD \rfloor$ times.

Suppose $\gamma$ has translation length $l.$ Then it crosses $\mA$ at least $\lfloor \frac lD \rfloor$ times, and each crossing contributes at least $c$ to the Margulis invariant of $u$. Therefore, $$\alpha_u (\gamma) \ge c \frac lD,$$ and by Proposition \ref{GT}, the action of $(\Gamma, u)$ on $\bR^{2n,2n-1}$ is proper.
\endproof

\section{Fundamental domains}\label{domains}
In this section, we prove properness of the cocycles constructed in Section \ref{construction} by building a fundamental domain for the associated affine action bounded by piecewise linear hypersurfaces in $\bR^{4n-1}$. This gives a second proof of properness which does not directly need the Ghosh--Treib criterion based on Margulis invariants. The hypersurfaces are boundaries of \emph{crooked halfspaces} introduced in \cite{bt2022}, adapted to the affine setting by taking the cone over a halfspace in the sphere of directions. This notion is a generalization of crooked planes in three dimensions defined by Drumm and Goldman, see \cite{Drumm, DrummGoldman2}.

We start by recalling the classical notion of \emph{sign variation} of a vector, which is used in order to define crooked halfspaces in higher dimensions. Then, we prove several basic properties of crooked halfspaces. The most important property is Proposition \ref{prop:stemquadrant}, which gives a precise description of which translations in $\bR^{4n-1}$ preserve a given crooked halfspace. Surprisingly, these translations form a $2$-dimensional convex cone in all dimensions. After proving a few accessory lemmas, we prove Theorem \ref{thm:domain}, which essentially states that the data of a positive representation $\rho:\Gamma \to \SO(2n,2n-1)$ and a generalized infinitesimal strip deformation $u: \Gamma \to \bR^{4n-1}$ allow us to ``lift" a fundamental domain in $\bH^2$ bounded by pairwise disjoint geodesic arcs to a fundamental domain in $\bR^{4n-1}$ bounded by pairwise disjoint crooked planes.

\begin{defn}
    The \emph{upper sign variation} of a nonzero vector $\vec{v}\in \bR^{2n,2n-1}$, denoted $S_E^+(\vec{v})$, is the number of times the sign of the coordinates of $\vec{v}$ in the basis $E$ change, where we independently assign a sign to the $0$ coordinates in order to maximize this value.
    
    The \emph{lower sign variation} $S_E^-(\vec v)$ is the number of sign variations in the coordinates of $\vec v$ in the basis $E$, assigning the signs to $0$ that minimize the value. Equivalently, it's the sign variation of a shorter vector where we delete the zero coordinates.
\end{defn}

For example, let $\vec v = (1,0,3)$ in the basis $E.$ Then $S^+_E(\vec v) = 2$ and $S^-_E(\vec v) = 0.$ For $\vec v' = (-1,1,2),$ the lower and upper sign variations are both $1.$

The upper and lower sign variations follow the following relationship:
\begin{equation}\label{eqn:signvaridentity}
    S^+_E(\vec{v}) + S^-_E(alt_E(\vec{v})) = 4n-2,
\end{equation}
where $alt_E(\vec{v})$ is the linear map given by a diagonal matrix in the basis $E$ with $i$th diagonal entry $(-1)^i$. This is clear for vectors with no zero coordinates. If some coordinates of $\vec{v}$ are zero, we consider the set $V$ of all ways of assigning $\pm 1$ to the zero coordinates of $\vec{v}$. The map $alt_E$ acts on $V$ bijectively, and for each element $\vec{u}\in V$ we have the equality (\ref{eqn:signvaridentity}). In particular, this implies that $S_E^\pm(\vec{u})$ is maximized when $S^\pm(alt_E(\vec{u}))$ is minimized and vice versa, yielding $(\ref{eqn:signvaridentity})$ for general $\vec{v}$.

Given any $\vec{v} \in \bR^{2n,2n-1}$, define $\sgn_E^+(\vec{v})$ as the set of all vectors obtained by replacing the zero coordinates of $\vec{v}$ with $\pm 1$ and attaining the maximum possible variation. Define $\sgn_E^-(\vec{v})$ similarly for minimizing the variation. In the case above, we have $\sgn_E^+((1,0,3)) = \{(1, -1, 3)\}$ and $\sgn_E^-((1,0,3)) = \{(1, 1, 3)\}.$ Note that in general the sets $\sgn_E^\pm$ can contain more than one element. 

Note that for any vector $\vec{v}\neq 0$, if the last several coordinates of $\vec{v}$ are zeroes, there is a unique choice of sign which minimizes (respectively maximizes) the sign variation for these coordinates (and similarly for the first several coordinates). We will denote the unique sign given to the $i$th coordinate by $\sgn_E^-(\vec{v})_i$ (resp. $\sgn_E^+(\vec{v})_i$) whenever this applies.

The following definition is adapted from Definition 5.11 in \cite{bt2022} to give a crooked halfspace in a vector space instead of a sphere.

\begin{defn}
    The \emph{open crooked halfspace} associated to an isotropic basis $E$ is the following open subset of $\bR^{2n,2n-1}$:
    \begin{align*}
    \half_E = \{\vec{v}\in \bR^{2n,2n-1} ~|~ S^+_E(\vec{v}) \le 2n-1 &\text{ and in case of equality,}\\
    &\sgn^+_{E}(\vec{v})_{4n-1} = +1\}.
    \end{align*}
    The \emph{closed crooked halfspace} $\overline{\half}_E$ is defined similarly:
    \begin{align*}
    \overline{\half}_E = \{\vec{v}\in \bR^{2n,2n-1} ~|~ S^-_E(\vec{v}) \le 2n-1 &\text{ and in case of equality,}\\
    &\sgn^-_E(\vec{v})_{4n-1} = +1\}.
    \end{align*}
    The equality case condition for $\overline \half_E$ is equivalent to requiring that the last nonzero coordinate of $\vec{v}$ in the basis $E$ is positive.
\end{defn}

\begin{lem}
    The closed crooked halfspace is the topological closure of the open crooked halfspace.
\end{lem}
\begin{proof}
    We first show that $\overline{\half}_E$ is closed. Let $\vec{v} \in \overline{\half}_E$. Choose $\vec{u} \in \sgn_E^-(\vec{v})$, so that $\vec{u}$ has no zero coordinates and $S^-(\vec{u}) = S^-(\vec{v})\le 2n-1$. Note also that the whole open orthant $O$ containing $\vec{u}$ is contained in $\overline{\half}_E$, and $\vec{v}\in \overline O$. This shows that $\overline{\half}_E$ is contained in a finite union of closures of orthants $\overline{O}$. Conversely, if $O$ is any open orthant contained in $\overline{\half}_E$, and $\vec{v}\in \partial O$, then $\vec{v}$ is obtained from a vector in $O$ by changing some nonzero coordinates to zero, which preserves the value of $S^-_E$. Moreover, if $S^-_E$ is equal to $2n-1$ on $O$, then by definition the last coordinate of vectors in $O$ is positive. For $\vec{v}$, there are two cases: either its last nonzero coordinate is positive, so $\vec{v}\in \overline{\half}_E$, or it is negative, in which case we must have $S^-_E(\vec{v})<2n-1$ and again $\vec{v}\in \overline{\half}_E$.

    Next, notice that any orthant contained in $\overline{\half}_E$ is also contained in $\half_E$, because in an open orthant $S^+_E = S^-_E$ and the second defining condition for the halfspaces also matches when there are no zero coordinates. This implies that any vector in $\overline{\half}_E$ can be approximated by a sequence of vectors in an open orthant which is contained in $\half_E$.
    
    Finally, we show that $\half_E \subset\overline \half_E$. This is because $\half_E$ is also contained in the same union of closures of orthants, since we observed before that the conditions for open and closed halfspaces match on vectors without zero coordinates, and moreover a vector $\vec{v}\in \half_E$ can be approximated by vectors in any of the open orthants represented by $\sgn^+_E(\vec{v})$.

\end{proof}

When $n=1$, an open crooked halfspace is obtained by taking the interior of the closure of the union of $4$ of the $8$ open orthants in $\bR^3$, corresponding to the signs $(+,+,+)$, $(-,+,+)$, $(-,-,+)$, $(-,-,-)$. In general, a crooked halfspace contains half of the orthants in its interior and is contractible.

We see from this example that the name \emph{halfspace} seems justified, and indeed in general we have the following lemma:
\begin{lem}\label{lem:opposite_halfspace}
    The complement of the open halfspace $\half_E$ is the closed halfspace defined by the opposite basis, $\overline{\half}_{\hat E}.$
\end{lem}
\proof If $\vec v = (v_1, v_2 , \ldots,v_{4n-2}, v_{4n-1})$ in basis $E,$ its expression in the basis $\widehat{E}$ is $(v_{4n-1}, - v_{4n-2}, \ldots, -v_2, v_1).$ We will denote the change of basis from $E$ to $\widehat{E}$ by $P$. By the relation in Equation \ref{eqn:signvaridentity} and the fact that reversing the order of the coordinates preserves sign variation, we have $S_E^+(\vec{v}) + S_{\widehat{E}}^-(\vec{v}) = 4n-2$. Therefore, $S_E^+(\vec v) \le 2n-1$ if and only if $S_{\hat E}^-(\vec v) \ge 2n-1.$ A second useful observation is that $\sgn^+_E(\vec{v}) = P\sgn^-_{\widehat{E}}(\vec{v})$.

The condition $\vec{v} \notin \half_E$ is equivalent to: $S^+_E(\vec{v}) \ge 2n-1$ and in case of equality $\sgn_E(\vec{v})_{4n-1} = -1$. Using the two observations above, this condition can be rewritten as: $S^-_{\widehat{E}}(\vec{v}) \le 2n-1$ and in case of equality $\sgn_{\widehat{E}}(\vec{v})_1 = -1$. Since $2n-1$ is odd, a vector having $2n-1$ sign variations has opposite signs for its first and last coordinate. We can therefore rewrite $\sgn_{\widehat{E}}(\vec{v})_1 = -1$ as $\sgn_{\widehat{E}}(\vec{v})_{4n-1} = +1$, and obtain precisely the condition for $\vec{v}\in \overline{\half}_{\widehat{E}}$.

\endproof

The boundary $\overline{\half}_E \smallsetminus \half_E = \overline{\half}_{\hat E} \smallsetminus \half_{\hat E}$ of a closed crooked halfspace will be called a \emph{crooked hyperplane} and denoted by $\mathcal C_E.$ 

\begin{defn}
    The \emph{stem-quadrant} of a halfspace is the set of vectors which translate $\half_E$ inside itself:
    \[SQ(E) = \{\vec{u}\in \bR^{2n,2n-1} ~|~ \vec{u} + \half_E \subseteq \half_E\}.\]
    Note that since $\half_E = (\overline{\half}_E)^\circ$ where $^\circ$ denotes the topological interior, the stem quadrant is also the set of vectors which translate $\overline{\half}_E$ inside itself.
\end{defn}

This is the same as the original definition of the stem quadrant (or \emph{translational semigroup}) for a crooked plane in $\bR^{2,1}$ introduced in \cite{BCDG}.

\begin{prop}\label{prop:stemquadrant}
    The stem-quadrant $SQ(E)$ is the convex cone generated by the vectors $-\vec{e}_1$ and $\vec{e}_{4n-1}$.
\end{prop}

\proof Let $\vec v \in \half_E$ and let $\vec{u} = -\alpha
\vec{e}_1 + \beta
\vec{e}_{4n-1}$ with $\alpha,\beta \ge 0.$
If $\vec v$ has negative first coordinate and positive last coordinate, so does $\vec v + \vec u,$ and $S_E^+(\vec v + \vec u) = S_E^+(\vec v)$ and therefore $\vec u + \vec v \in \half_E.$

If the first and last coordinates of $\vec v$ are both positive, $S_E^+$ is even, so $S_E^+(\vec v) \ne 2n-1$. Then $\vec v + \vec u$ might have negative first coordinate, in which case $S_E^+(\vec v + \vec u) = S_E^+(\vec v) + 1 \le 2n-1$ with last coordinate positive, so $\vec v + \vec u \in \half_E.$ A similar argument holds when $\vec v$ has negative first and last coordinate.

If the first coordinate of $\vec v$ is positive and the last is negative, $S_E^+(\vec v)$ is odd and not equal to $2n-1,$ so $S_E^+(\vec v) \le 2n-3.$ If adding $\vec u$ to $\vec v$ changes just the sign of first coordinate, the sign change increases by at most $1.$ If adding $\vec{u}$ to $\vec{v}$ changes the sign of both the first and last coordinate, then the sign variation increases by at most $2$. But if it increases by $2$, note that $\vec v + \vec u$ has positive last coordinate and $S^+_E(\vec{v}+\vec{u}) = 2n-1$ is allowed to hold, so $\vec v + \vec u \in \half_E.$

If the first or last coordinate of $\vec{v}$ vanishes, use the sign assigned to that coordinate when computing $S^+_E$ and the same arguments as above.

For the reverse inclusion, let $\vec{u}\in SQ(E)$ and assume first $\vec{u}_1 >0$. We let $\vec{a} \in \half(E)$ be a vector with very large coordinates $\vec{a}_2<0,\vec{a_3}>0,\dots,\vec{a}_{2n+1}>0$ and all other $\vec{a}_i$ small and positive. Then, $\vec{u}+\vec{a}$ must have alternating signs in its first $2n+1$ coordinates and therefore $S^+(\vec{u}+\vec{a})\ge2n$ a contradiction. An identical argument using the last $2n+1$ coordinates yields a contradiction if $\vec{u}_{4n-1}<0$.

Suppose now that there is some index $i \in \{2, 3, \ldots, 4n-3,4n-2\}$ and some sign $\sigma \in \{-,+\}$ such that $\vec u_i$ is non-zero and has sign $\sigma.$ We will show that $\vec u$ is not in $SQ(E).$

Let $\vec a \in \half(E)$ be a vector with $2n-2$ sign changes and no zero coordinates in the basis $E$, and with the signs of $\vec a_{i-1}, \vec a_i, \vec a_{i+1}$ all equal to $-\sigma.$ This can be achieved by a vector with all the sign changes occurring in the first $2n-1$ coordinates if $i \ge 2n$, or with all the sign changes occurring in the last $2n-1$ coordinates if $i < 2n.$ Since the sign variation of $\vec a$ is not $2n-1$, the sign of the last coordinate does not matter. Further, let $|\vec a_j| > |\vec u_j|$ for all $j \ne i$ and $|\vec a_i| < |\vec u_i|.$

This condition guarantees that $\vec a + \vec u$ has no zero coordinates. Crucially, it also means that for $j \ne i,$ the signs of  $\vec a_j + \vec u_j$ are the same as the signs of $\vec a_j,$ whereas the sign of $\vec a_i + \vec u_i$ is $\sigma.$ Therefore, the vector $\vec a + \vec u$ has two more sign changes than $\vec a$ from the new sign changes occurring between the coordinates $i-1$ and $i$ and between the coordinates $i$ and $i+1.$ Thus $S^\pm_E(\vec a + \vec u) = 2n > 2n-1$ and $\vec a + \vec u \notin \overline\half(E),$ meaning that $\vec u \notin SQ(E).$
\endproof

There is an alternate description of a crooked halfspace, similar to crooked halfspaces in \cite{bt2022}, Section 5.2.

For a flag $F \in \Flag^+(\bR^{4n-1}),$ define the \emph{positive half-subspace} as $$F_+^{(2n)} = \{\vec{v} \in \bR^{2n,2n-1} ~|~ F^{(2n-1)} \oplus \vec{v} = F^{(2n)} \text{ as oriented vector spaces}\}.$$

\begin{prop}\label{prop:quadruples}
    Let $F_1, F_2$ be oriented flags and let $E$ be the isotropic basis associated to them so that $F_E = F_1$ and $F_{\hat E} = F_2.$ Then
    $$\half_E = \bigcup_{F \in \ival{F_1}{F_2}}F^{(2n)}_+,$$
    and
    $$\overline{\half}_E = \bigcup_{F \in \overline{\ival{F_1}{F_2}}}\overline{F^{(2n)}_+}.$$
\end{prop}
\begin{proof}
    The first equality follows directly from \cite{bt2022} Proposition 5.15. For the second equality, let $\vec{v}_j \rightarrow \vec{v}$ be a convergent sequence in the right hand side. Then, $\vec{v}_j \in \overline{G_{j+}^{(2n)}}$ for some $G_j \in \overline{\ival{F_1}{F_2}}$. Choose $K\subset \SL(4n-1,\bR)$ a compact neighborhood of the identity such that $KG_0 = \overline{\ival{F_1}{F_2}}$. Then, there exist a sequence $k_j \in K$ such that $k_j G_0 = G_j$, and up to taking a subsequence $k_j \rightarrow k \in K$. We then have $k_j^{-1} \vec{v}_j \in \overline{G_{0+}^{(2n)}}$ for all $j$, and since this is a closed subset the limit $k^{-1} \vec{v} \in \overline{G_{0+}^{(2n)}}$, implying that $\vec{v} \in k \overline{G_{0+}^{(2n)}}$. We conclude that the right hand side is a closed set. Clearly this set contains $\half_E$, so it contains $\overline{\half_E}$.

    For the reverse inclusion, let $\vec{v} \in \overline{F_{+}^{(2n)}}$ for some $F\in \overline{\ival{F_1}{F_2}}$. Choose a sequence $\vec{v}_j \to \vec{v}$ with $\vec{v}_j \in F_{+}^{(2n)}$. Using the same strategy as above, write $\overline{\ival{F_1}{F_2}} = K F$ for some $K$ a compact subset. Then, there exists a sequence $k_j \in K$ converging to the identity such that $k_j F \in \ival{F_1}{F_2}$ for all $j$. We deduce that $k_j \vec{v}_j \to \vec{v}$, and by construction $k_j\vec{v}_j \in \half_E$ for all $j$ so $\vec{v} \in \overline{\half_E}$.
\end{proof}

We will also denote $\half_E = \half(F_E, F_{\hat E}),$ $\mathcal C_E = \mathcal C(F_E, F_{\hat E})$, and $SQ(E) = SQ(F_E,F_{\widehat{E}})$. Note that $\half_{\hat E} = \half(F_{\hat E}, F_E).$ It also follows from this definition that $g\half_E = \half_{gE}.$

\begin{rem}
    Note that, while the endpoints $F_1,F_2$ are isotropic flags, we have to take the union over the interval in the space of \emph{full} flags, not only isotropic ones. Even for $n=1,$ a union over only the isotropic flags would not cover timelike points in $\mathbb R^{2,1},$ see illustration in Figure \ref{fig:intervals}.
\end{rem}
\begin{figure}[h!]
    \centering
    \includegraphics[width=0.4\linewidth]{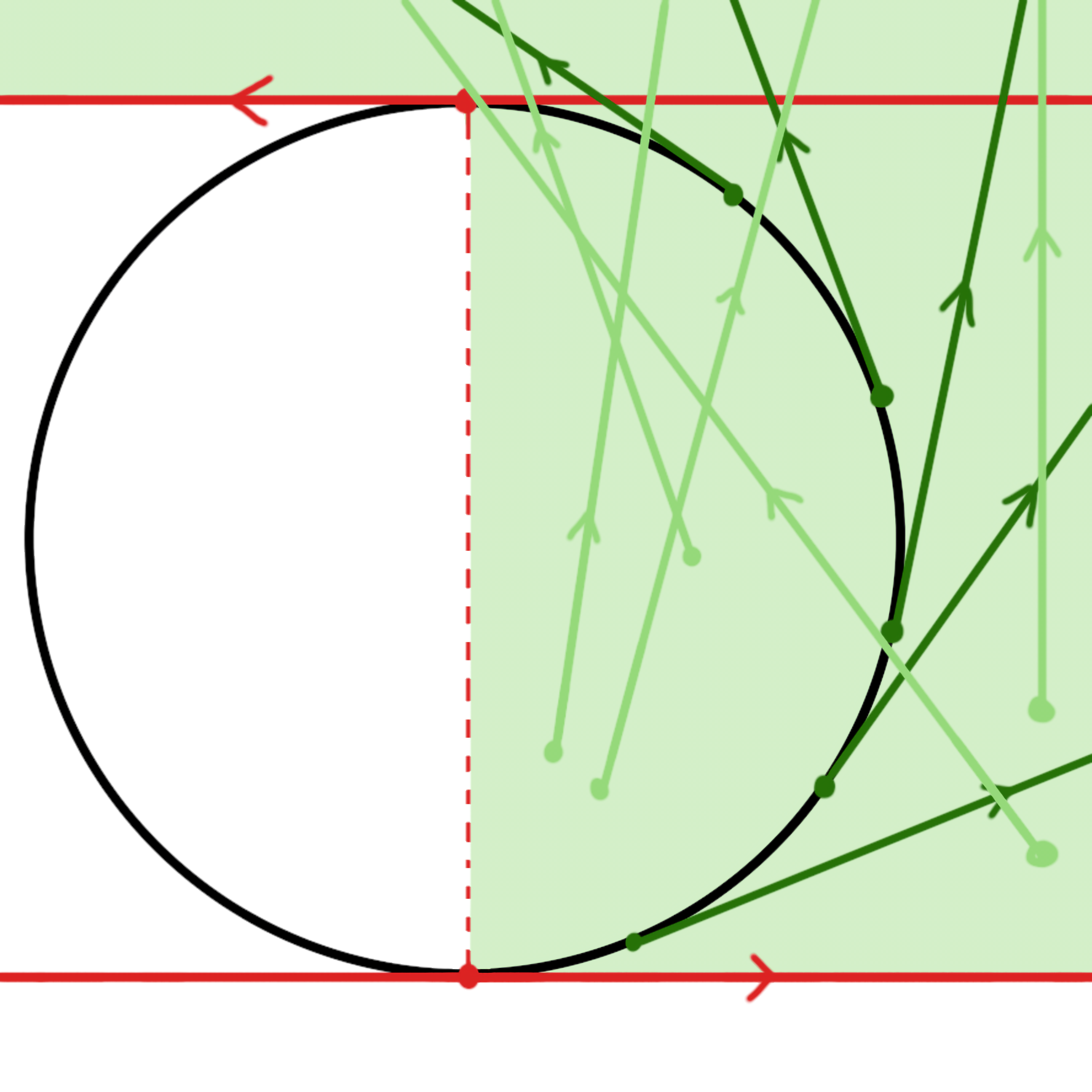}
    \includegraphics[width = 0.4\linewidth]{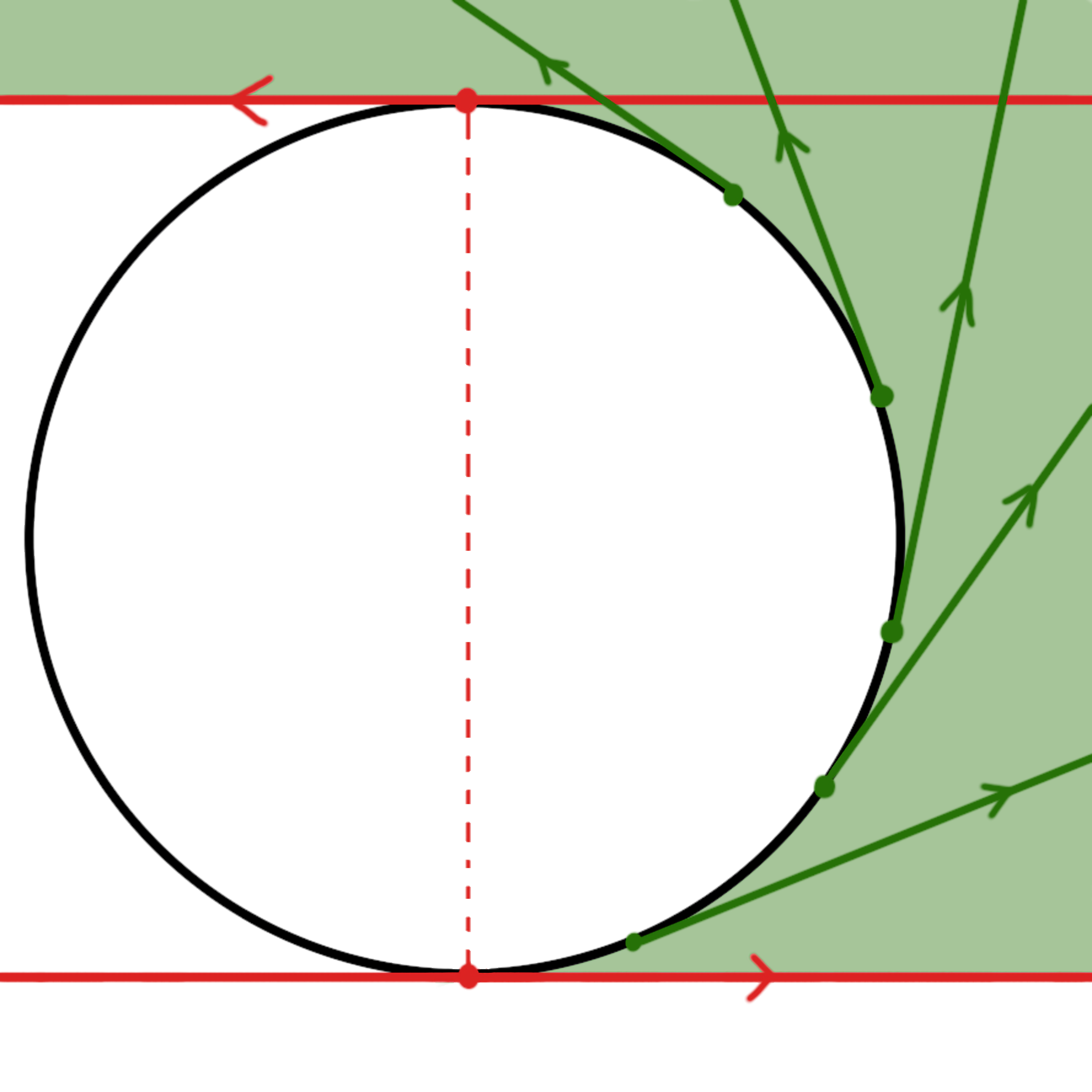}
    \caption{On the left, a cross-section of $\bigcup_{F\in ((F_1,F_2))}F_+^{(2)} \subset \bR^{2,1}$, where $F_1,F_2$ are depicted in red and
    the half-subspaces in dark green are the isotropic ones.  On the right, we take the union over only isotropic flags. No halfspace defined this way includes any  timelike vectors.}
    \label{fig:intervals}
\end{figure}

\begin{prop}\label{prop:halfspace_inclusion}
    Let $F, G, G', F'$ be a positive quadruple of flags. Then, the following inclusion and disjointness properties hold.
    \begin{enumerate}
        \item $\half(G, G') \subset \half(F,F')$;
        \item $\overline{\half}(F, G)\cap\overline{\half}(G', F') = \{0\}$, and in particular $\half(F, G)$ is disjoint from $\half(G', F')$;
    \end{enumerate}
\end{prop}
\proof \begin{enumerate}
    \item By Lemma \ref{lem:intervals_proper}, $\overline{\ival{G}{G'}} \subset \ival{F}{F'}$ and so the first claim follows from Proposition \ref{prop:quadruples}.
    \item Any pair of flags $X\in \overline{\ival{F}{G}}$, $Y\in \overline{\ival{G'}{F'}}$ is oriented-transverse. This implies that the subspaces $\overline{X_+^{(2n)}}$ and $\overline{Y_+^{(2n)}}$ intersect only in $0$. Indeed, if $E$ is a positive basis such that $F_E = X$ and $F_{\widehat{E}} = Y$, then $X^{(2n)} \cap Y^{(2n)} = \langle \vec{e}_{2n} \rangle,$ with $\vec{e}_{2n} \in X_+^{(2n)}$ but not in $Y_+^{(2n)}.$ 
\end{enumerate} \endproof 

In $\bR^3$, if two closed crooked halfspaces only intersect at the origin, it suffices to translate each of them by a vector in the interior of its own stem quadrant in order to make them disjoint. In higher dimensions, this is not true anymore since for each $\vec{v}\in SQ(E)^\circ$, the opposite vector $-\vec{v}$ actually lies in the boundary of the closed crooked halfspace $\overline{H}_E$, so the two halfspaces will still intersect at the origin\footnote{We thank the anonymous referee for pointing this out, fixing a gap in a previous version of this paper.}. The following proposition shows that this can be fixed by adding a third crooked plane in between the two crooked halfspaces.

\begin{prop}\label{prop:halfspace_separated_inclusion}
    Let $F, G, H, H', G', F'$ be a positive configuration of six flags in $\bR^{4n-1}$ and let $\vec{u}\in SQ(G,G')\smallsetminus\{0\}, \vec{v}\in SQ(H,H')\smallsetminus\{0\}$. Then,
    \[\vec{u}+\vec{v}+\overline{\half}(H,H')\subset \half(F,F'),\] or equivalently $\left(\vec{u}+\vec{v}+\overline{\half}(H,H')\right)\cap\overline{\half}(F',F) = \emptyset$.
\end{prop}
\begin{proof}
    We work in a basis $E$ such that $F=F_E$ and $F'=F_{\widehat{E}}$.

    By Proposition \ref{prop:halfspace_inclusion}, the untranslated closed halfspaces satisfy $\overline{\half}(G,G') \smallsetminus \{0\} \subset \half(F,F')$. Moreover, by definition of stem quadrants we have
    \[\vec{u} + \vec{v} + \overline{\half}(H,H')\subset \vec{u} + \overline{\half}(H,H') \subset \vec{u} + \overline{\half}(G,G') \subset \overline{\half}(G,G').\]

    Therefore, for the Proposition to hold it suffices to show that $0 \notin \vec{u}+\vec{v}+\overline{\half}(H,H')$.
    
    The positivity of the tuple of flags implies that there exists a pair of totally positive matrices $M_1,M_2$ such that $M_1 F = G, M_1 F' = G'$ and $M_1M_2 F = H$, $M_1M_2F' = H'$. Equivariance of stem quadrants then implies that we can write $\vec{u} = M_1 \vec{u}'$ and $\vec{v} = M_1 M_2 \vec{v}'$ for some $\vec{u}',\vec{v}'\in SQ(E)$, the standard stem quadrant.

    Suppose for the sake of contradiction that there exists $\vec{w}\in \overline{\half}(H,H')$ such that $\vec{u}+\vec{v}+\vec{w} = 0$. We can write $\vec{w} = M_1M_2 \vec{w}'$ for some $\vec{w}' \in \overline{\half}_E$. Then, we have $M_1\vec{u}' + M_1M_2\vec{v}' + M_1M_2\vec{w}' = 0$ which implies $-\vec{u}' = M_2(\vec{v}' + \vec{w}')$. Note that $\vec{v}'+\vec{w}'\neq 0$ since $\vec{u}\neq 0$ by hypothesis. By the variation diminishing property of totally positive matrices (\cite{Pinkus}, Theorem 3.3) and the fact that $\vec{v}' + \vec{w}'\in \overline{\half}_E$,
    \[S^+_E(M_2(\vec{v}' + \vec{w}')) \le S^-_E(\vec{v}' + \vec{w}') \le 2n-1.\]
    On the other hand, $S^+_E(-\vec{u}') \ge 4n-3$ since only the first and last coordinates of $\vec{u}'$ are nonzero. The only way these inequalities can be simultaneously true is if $n=1$, in which case $S^+_E(-\vec{u}')=1$. By Proposition \ref{prop:stemquadrant}, this implies that $-\vec{u}' = a \vec{e}_1 - b\vec{e}_3$ for some $a,b>0$, but then $-\vec{u}' \notin \overline{\half}_E$ since it does not satisfy the sign condition for the equality case, a contradiction.
\end{proof}

We now construct a fundamental domain for $(\rho(\Gamma), u)$ acting on $\bR^{2n,2n-1},$ with $u$ defined by an arc system $\mA$
as in Section \ref{construction}. Our strategy is to use appropriately translated crooked hyperplanes to bound the domain.

We will mimic the construction of a cocycle from Section \ref{construction}. Let $a \in \tilde \mA$ be an arc in $\bH^2.$ Let $p_a$ be a point on the arc $a$. Let $c_a \colon [0,1] \to \bH^2$ be a path with $c_a(0) = p_0$ and $c_a(1) = p_a,$ with no self-intersections in any arc in $\tilde \mA$, and such that all intersections with an arc are transverse. Denote by $X_{c_a} $
the set of intersection points between $c_a([0,1])$ and the geodesics in $\tilde\mA.$ For $x\in X_{c_a}$, denote by $a_x$ the arc in $\tilde\mA$ intersected by $c_a([0,1])$ at $x,$ and let $\sigma(x) = 1$ if the path $c_a$ crosses the arc $a_x$ positively at $x$, and $-1$ otherwise (for $p_a = c(1)$ use the left derivative  to determine orientation).
Define 
$$\tilde u(a) = \sum_{x\in X_{c_a}\smallsetminus\{p_a\}} \sigma(x)(
\vec v_{a_x}^+ - 
\vec v_{a_x}^-) + \frac12 \sigma(p_a)(
\vec v_{a}^+ - 
\vec v_{a}^-) .$$
Intuitively, $\tilde u(a)$ is the displacement vector between the region containing $p_0$ and the region containing $p_a.$ This definition is very similar to the definition of $u(\gamma)$ in \ref{def:cocycle}, but here $p_a$ is not necessarily a translate by $\Gamma$ of $p_0,$ and the very last ``contribution" to $\tilde u (a)$ is only counted with a factor of $\frac 12$.

For each $a \in \tilde \mA,$ define a translated crooked hyperplane $$\mathcal C(a) = \mathcal C(\xi(a^+), \xi(a^-)) + \tilde u(a)$$
and the translated crooked halfspace $$\half(a) = \half(\xi(a^+), \xi(a^-)) + \tilde u(a).$$

\begin{lem}\label{lem:crooked_equivariance}
    The families of crooked hyperplanes and crooked halfspaces satisfy the following equivariance properties:
    \[\mathcal C(\gamma \cdot a) = \rho(\gamma)\mathcal C(a) + u(\gamma);\]
    \[ \half(\gamma \cdot a) = \rho(\gamma) \half(a) + u(\gamma),\]
    for all $\gamma \in \Gamma$.
\end{lem}
\begin{proof}
    We will only show the equality \[ \half(\gamma \cdot a) = \rho(\gamma) \half(a) + u(\gamma),\] as the other equality follows. 

    Let $\gamma \in \Gamma$ and $a \in \tilde\mA.$ Fix a point $p_a \in a$. Then $\gamma\cdot p_a$ is a point on $\gamma \cdot a;$ call it $p_{\gamma \cdot a}.$

    Let $c_\gamma:[0,1] \to \bH^2$ be a path joining the basepoint $p_0$ to $\gamma \cdot p_0,$ thus defining $u(\gamma)$ and let $c_a: [0,1] \to \bH^2$ be a path defining $\tilde u(a).$ Then the concatenated path $c:[0,1] \to \bH^2$ defined by
    $$c(t) = \begin{cases}
        c_\gamma(2t); & 0 \le t \le \frac{1}{2} \\
        \gamma \cdot c_a(2t-1); & \frac12 \le t \le 1
    \end{cases}$$
    is a continuous path starting at $p_0,$ with $c(\frac12) = \gamma\cdot p_0,$ and $c(1) = \gamma \cdot p_a = p_{\gamma \cdot a}.$ Therefore, $c$ defines $\tilde u(\gamma \cdot a)$. Note that by construction, the first half of the path $c$ defines $u(\gamma),$ and by equivariance of $\vec v_{a'}^+ - \vec v_{a'}^-,$ the second half of $c$ defines $\rho(\gamma)\tilde u(a).$

    Therefore, 
    \begin{eqnarray*}
        \half(\gamma \cdot a) &=& \half( \xi((\gamma \cdot a)^+), \xi((\gamma \cdot a)^-) ) + \tilde u(\gamma \cdot a)\\
        &=& \rho(\gamma) \half(\xi(a^+), \xi(a^-)) + \tilde u(\gamma \cdot a) \\
        &=& \rho(\gamma) \half (\xi(a^+), \xi(a^-)) + u(\gamma) + \rho(\gamma ) \tilde u(a) \\
        &=& \rho(\gamma)\left(\half (\xi(a^+), \xi(a^-)) + \tilde u(a)\right) + u(\gamma) \\
        &=& \rho(\gamma) \half(a) + u(\gamma).
    \end{eqnarray*}
    Here, the first nontrivial equality follows from equivariance of the usual crooked halfspaces, and the rest from the above discussion.
\end{proof}

\begin{lem}\label{lem:disjointness}
    Let $a, a'$ be different non-adjacent arcs in $\tilde \mA.$ Then $\mC(a)$ and $\mC(a')$ are disjoint.
\end{lem}
\proof Consider first the case when $a$ and $a'$ are separated by only one arc $b$, and assume without loss of generality that from the base point we can cross $a$,$b$,$a'$ in the positive orientation, in that order, consecutively. Denote $\vec v = 
\frac12(\vec v_{a}^+ - 
\vec v_{a}^-$), $\vec{u} = 
\vec v_{b}^+ - 
\vec v_{b}^-$ and $\vec v' = \frac12 (
\vec v_{a'}^+ - 
\vec v_{a'}^-).$ We can write the translated crooked hyperplanes as $\mC(a) = \mC(\xi(a^+), \xi(a^-)) + \tilde u(a)$ and 
$$\mC(a') = \mC(\xi(a'^+), \xi(a'^-)) + \tilde u(a) + \vec{v}+\vec{u}+\vec{v}'.$$ 
 Note it is therefore enough to show that $\mC(\xi(a^+), \xi(a^-))$ and $\mC(\xi(a'^+), \xi(a'^-)) +\vec v + \vec{u} + \vec v'$ are disjoint. 

By Proposition \ref{prop:halfspace_separated_inclusion}, and since $\vec{u},\vec v'$ are respectively in the stem-quadrants of $\half(\xi(b^-), \xi(b^+))$ and $\half(\xi(a'^-), \xi(a'^+)),$ as characterized in Proposition \ref{prop:stemquadrant}, we have 
\begin{eqnarray*}
 \mC((\xi(a'^-), \xi(a'^+)) &=& \partial (\overline{\half(\xi(a'^-), \xi(a'^+))}) +\vec v'+ \vec{u} + \vec v\\
&\subset & \overline{\half(\xi(a'^-), \xi(a'^+))} + \vec v' + \vec{u} + \vec v \\
 & \subset & \half(\xi(a^-), \xi(a^+)) + \vec v \subset \half(\xi(a^-), \xi(a^+)),
\end{eqnarray*}
with $\half(\xi(a^-), \xi(a^+))$ being disjoint from $\mC(\xi(a^-), \xi(a^+))$.

The general statement follows by induction.
\endproof

The following definitions and the next two technical lemmas will be used in the proof of Theorem \ref{thm:domain} in order to show that the translates of our candidate fundamental domain actually tile all of $\bR^{4n-1}$.

For $F,G$ a pair of oriented, transverse isotropic flags, define
\[\ival{F}{G}^{\operatorname{isot}} = \{ X \in \ival{F}{G} ~|~ X \text{ is isotropic.}\}\]

 Recall  Definition \ref{def:neutral_vec} of $\vec x^0(X, Y)$ for $X$ and $Y$ isotropic; it is the normalized middle vector of the basis adapted to $X$ and $Y.$ Note that it suffices for the middle vector of this basis to be spacelike to make this definition. We will thus extend $\vec x^0 $ to pairs of transverse flags where only one is isotropic, preserving the property of being unit spacelike.

\begin{defn}
    
For $Y$ an isotropic flag and $X$ transverse to $Y$, define $\vec x^0(X, Y)^*: \bR^{4n-1} \to \bR$ by writing $$\vec v =\vec x + \left(\vec{x}^0(X, Y)^*(\vec v)\right)\vec{x}^0(X, Y) + \vec y $$ with $\vec x \in X^{(2n-1)}, \vec y \in Y^{(2n-1)}.$\end{defn}
As the notation suggests, this function associates to $\vec{v}$ the coefficient of $\vec{x}^0(X,Y)$ in a positive basis adapted to the pair $X,Y$ extending this unit spacelike vector.

\begin{lem}\label{lem:uniform_bound_coordinate}
    Let $F_1,F_2,F_3,G_3,G_2,G_1$ be a positive $6$-tuple of isotropic oriented flags, and fix $f,g$ positive vectors in the $1$-dimensional parts of $F_2,G_2$ respectively. Then, there is a uniform lower bound
    \[\vec{x}^0(X,Y)^*(f-g) > m_{F_1,F_3,G_3,G_1}^{f,g}\]
    for some constant $m_{F_1,F_3,G_3,G_1}^{f,g}>0$, where $X\in \ival{F_3}{G_3}$ and $Y\in\ival{G_1}{F_1}^{\operatorname{isot}}$. Moreover, this bound is invariant under the simultaneous action of $\SO(2n,2n-1)$ on $F_1,F_3,G_1,G_3,f,g$.
\end{lem}
\begin{proof}
    By Lemma \ref{Positivity}, $\vec{x}^0(X,Y)^*(f-g) > 0$.

    Now, the closures $\overline{\ival{G_1}{F_1}^{\operatorname{isot}}}$ and $\overline{\ival{F_3}{G_3}}$ are compact and so the continuous function $(X,Y)\mapsto \vec{x}^0(X,Y)^*(f-g)$ has a nonnegative minimum value $m_{F_1,F_3,G_3,G_1}^{f,g}$. This minimum value cannot be $0$ as the function is also positive on a pair of slightly larger intervals respectively containing $\overline{\ival{G_1}{F_1}^{\operatorname{isot}}}$ and $\overline{\ival{F_3}{G_3}}$ in their interiors.

    Let $A\in \SO(2n,2n-1)$. Then, $\vec{x}^0(AX,AY)^*(Af-Ag) = \vec{x}^0(X,Y)^*(f-g)$. It follows that the minimum found above satisfies $m_{AF_1,AF_3,AG_3,AG_1}^{Af,Ag}=m_{F_1,F_3,G_3,G_1}^{f,g}$.
\end{proof}

Let $\widetilde{\mA}$ be a collection of pairwise disjoint arcs in $\bH^2$. We say that three arcs $a_1,a_2,a_3\in \widetilde{\mA}$ are \emph{consecutive} if they are distinct and can be crossed transversely, in order, by a finite-length path, without any other intersection with arcs in $\widetilde{\mA}$ (including $a_i$).

\begin{lem}\label{lem:uniform_on_surface}
    Let $\rho: \Gamma \to \SO(2n,2n-1)$ be a positive Anosov representation of a free group $\Gamma < \PSL_2(\bR)$ and let $\mA$ be a filling system of arcs on $ \bH^2/\Gamma.$ Let $a_1, a_2, a_3$ be three consecutive arcs in $\bH^2$ that are lifts of arcs of $\mA.$ Then for all $X \in ((\xi(a_3^+), \xi(a_3^-)))$ and $Y \in((\xi(a_1^-), \xi(a_1^+)))$, the estimate  $\vec x^0(X, Y)^*(
    \vec{v}^+_{a_2} - 
    \vec{v}^-_{a_2}) \geq m > 0$ holds, with $m$ independent of $a_1, a_2, a_3, X, Y.$
\end{lem}
\proof By Lemma \ref{lem:uniform_bound_coordinate}, we know that $\vec x^0(X, Y)^*(
\vec v^+_{a_2} - 
\vec v^-_{a_2}) \geq C_{a_1,a_2,a_3} > 0$ for  $C_{a_1,a_2,a_3}:= m_{\xi(a_1^+),\xi(a_3^+),\xi(a_3^-),\xi(a_1^-)}^{\vec{v}_{a_2}^+,\vec{v}_{a_2}^-} .$ However, up to the action of $\Gamma < \SO(2n,2n-1),$ there are only finitely many configurations of three consecutive arcs. Because the quantity $\vec x^0(X, Y)^*(
\vec v^+_{a_2} - 
\vec v^-_{a_2})$ is invariant under the action of $\SO(2n,2n-1),$ that means there are only finitely many distinct $C_{a_1, a_2, a_3}.$ Define $$m = \min_{\{(a_1, a_2, a_3) \text{ consecutive arcs}\}}C_{a_1, a_2, a_3}>0.$$ 
\endproof 

Call a system of arcs $\mA$ \emph{overfilling} if the collection of lifts $\tilde \mA$ contains a subcollection of arcs bounding a fundamental domain, such that each pair of arcs bounding the fundamental domain is separated by at least one arc of $\tilde \mA$. Overfilling systems of arcs always exist for a convex-cocompact surface $\bH^2/\Gamma$ and can be constructed, for example, by taking a filling system of arcs and adding a parallel arc on each side of every arc (with distinct endpoints in $\partial \bH^2$).

Let $\mA$ be an overfilling system of arcs. Let $\mathcal K$ be a fundamental domain in $\bH^2$ for the action of $\Gamma$ bounded by a subcollection of arcs $a_1, a_1', \ldots, a_k, a_k'$ which are pairwise separated by some arc of $\tilde\mA$. Let $\gamma_i \in \Gamma$ be such that $\gamma_i a_i = a_i'.$ Then the $\gamma_i$ are a free generating set for $\Gamma.$ Define $\mathcal D$ to be the domain in $\bR^{2n,2n-1}$ bounded by the translated crooked hyperplanes $\mC(a_1), \mC(a_1'), \ldots, \mC(a_k), \mC(a_k').$

\begin{thm}\label{thm:domain}
    Let $\rho: \Gamma \to \SO(2n,2n-1)$ be a positive Anosov representation of a free group $\Gamma < \PSL_2(\bR)$ and let $\mA$ be an overfilling system of arcs on $ \bH^2/\Gamma.$ Let $u: \Gamma \to \bR^{2n,2n-1}$ be a generalized infinitesimal strip deformation obtained from this data. Then $\mathcal{D}$ is a fundamental domain for the action of $(\Gamma, u)$ on $\bR^{2n,2n-1}.$
\end{thm}

\proof By definition of $\mathcal D$ and equivariance (Lemma \ref{lem:crooked_equivariance}), we get that $(\Gamma, u)$ identifies the sides of $\mathcal D$.

Lemma \ref{lem:disjointness} and its proof guarantee that $\mathcal D$ and $\rho(\gamma) \cdot \mathcal{D} + u(\gamma)$ have disjoint interiors for every nontrivial $\gamma \in \Gamma.$

Suppose $(\Gamma, u) \cdot \mathcal{
D}$ is not all of affine space. Then there is some sequence $(a_k)_k$ of consecutive lifts of arcs, oriented away from $p_0$, and such that $\bigcap_k \half(a_k) \neq \emptyset.$ Let $p$ be an element in the intersection $\bigcap_k \half(a_k)$. By Proposition \ref{prop:quadruples}, we can write $p = x_k + \tilde u(a_k)$ with $x_k \in (X_k^{(2n)})_+$ for some $X_k$ in $((\xi(a_k^+), \xi(a_k^-)))$ for every $k.$ Choose an isotropic $Y \in ((\xi(a_1^-),\xi(a_1^+))).$

For each $k,$ we have $\vec x^0(X_k, Y)^*(x_k)>0,$ as $x_k \in (X_k^{(2n)})_+ = X_k^{(2n-1)}+\bR^+\vec x^0(X_k, Y).$

Denote by $\vec v_a = 
\vec v_a^+ - 
\vec v^-_a$ and estimate
\begin{align*}
\vec x^0(X_k, Y)^*(\tilde u(a_k)) &= \sum_{i = 1}^{k-1} \vec x^0(X_k, Y)^*(\vec v_{a_i}) + \frac12 \vec x^0(X_k, Y)^*(\vec v_{a_k}) \\
&= \vec x^0(X_k, Y)^*(\vec v_{a_1}) + \frac12 \vec x^0(X_k, Y)^*(\vec v_{a_k}) + \sum_{i = 2}^{k-1} \vec x^0(X_k, Y)^*(\vec v_{a_i}) \\
&\geq  0 + \sum_{i = 2}^{k-1} \vec x^0(X_k, Y)^*(\vec v_{a_i}) \\
&\geq (k-2)m,
\end{align*}
where the first inequality is due to Lemma \ref{Positivity} and the second inequality is due to Lemma \ref{lem:uniform_on_surface}, using the fact that $Y \in ((\xi(a_1^-),\xi(a_1^+) )) \subset ((\xi(a_{i-1}^-),\xi(a_{i-1}^+)))$ and $X_k \in ((\xi(a_k^+),\xi(a_k^-) )) \subset ((\xi(a_{i+1}^+),\xi(a_{i+1}^-)))$ for all $2 \leq i \leq k-1.$

We therefore have the estimate $$\vec x^0(X_k, Y)^*(p) = \vec x^0(X_k, Y)^*(x_k + \tilde u(a_k)) \geq (k-2)m$$ for all $k.$ However, up to passing to a subsequence of $X_k$, the functionals $\vec x^0(X_k, Y)^*$ converge to $\vec x^0(X,Y)^*$ where $X=\lim_{k\to\infty} X_k$, a contradiction.

\endproof

\printbibliography
\end{document}